\documentclass[reqno]{amsart}
\usepackage{caption}
\usepackage{subcaption}

\usepackage{mathabx,cite}
\usepackage[utf8]{inputenc} 
\usepackage[english]{babel}
\usepackage{amstext}
\usepackage{euscript}
\usepackage{bbm}
\usepackage{mathrsfs}
\usepackage{enumerate}
\usepackage{graphicx}
\usepackage{amscd}
\usepackage{verbatim}
\usepackage{amssymb}
\usepackage{amsthm}
\usepackage{amsmath}
\usepackage{amstext}
\usepackage{amsfonts}
\usepackage{latexsym}
\usepackage{mathtools} 
\usepackage{enumitem} 
\usepackage{accents} 
\usepackage[foot]{amsaddr} %package that puts the authors addresses on the title page as a footnote (when using AMSART document class only).
\usepackage{float}

 \usepackage[left=1.2in, right=1.2in, top=1.5in, bottom=1.5in]{geometry}
\usepackage{scalerel,stackengine}

\usepackage[usenames,dvipsnames]{xcolor} 
\definecolor{dkblue}{RGB}{1,31,91} % This is a dark Blue     
\usepackage[colorlinks=true, pdfstartview=FitV, linkcolor=dkblue, citecolor=dkblue, urlcolor=dkblue]{hyperref}

\newcommand{\RR}{\mathbb R}

\newcommand{\ZZ}{\mathbb Z}

\newcommand{\TT}{\mathbb T}

\newcommand{\pa}{\partial}

\newcommand{\Id}{\mathrm{Id}}

\normalsize
\normalsize
\newcommand{\norm}[1]{\left\lVert#1\right\rVert}    
\newcommand\abs[1]{\left|#1\right|}    

\newcommand\restr[2]{{% we make the whole thing an ordinary symbol
  \left.\kern-\nulldelimiterspace % automatically resize the bar with \right
  #1 % the function
  \vphantom{\big|} % pretend it's a little taller at normal size
  \right|_{#2} % this is the delimiter
  }}
  
  \stackMath
\newcommand\reallywidehat[1]{%
\savestack{\tmpbox}{\stretchto{%
  \scaleto{%
    \scalerel*[\widthof{\ensuremath{#1}}]{\kern-.6pt\bigwedge\kern-.6pt}%
    {\rule[-\textheight/2]{1ex}{\textheight}}%WIDTH-LIMITED BIG WEDGE
  }{\textheight}% 
}{0.5ex}}%
\stackon[1pt]{#1}{\tmpbox}%
}

\theoremstyle{definition}
\newtheorem{theorem}{Theorem}

\newtheorem{lemma}[theorem]{Lemma}

\newtheorem{remark}[theorem]{Remark}

\numberwithin{equation}{section}
\numberwithin{theorem}{section}
\numberwithin{definition}{section}

\begin{document}

\subjclass[2010]{}

\title[Asymptotic models for viscoelastic one-dimensional blood flow]{Asymptotic models for viscoelastic one-dimensional blood flow}

\author[D. Alonso-Or\'an]{Diego Alonso-Or\'an}
\address{Departamento de An\'{a}lisis Matem\'{a}tico and Instituto de Matem\'aticas y Aplicaciones (IMAULL), Universidad de La Laguna, C/Astrof\'{i}sico Francisco S\'{a}nchez s/n, 38271, La Laguna, Spain. \href{mailto:dalonsoo@ull.edu.es}dalonsoo@ull.edu.es}

\author[R. Granero-Belinch\'on]{Rafael Granero-Belinch\'on}
\address{Departamento  de  Matem\'aticas,  Estad\'istica  y  Computaci\'on,  Universidad  de Cantabria.  Avda.  Los  Castros  s/n,  Santander,  Spain. \href{mailto:rafael.granero@unican.es}rafael.granero@unican.es}

\author[Carlos Yanes P\'erez]{Carlos Yanes P\'erez}
\address{Universidad de La Laguna, C/Astrof\'{i}sico Francisco S\'{a}nchez s/n, 38271, La Laguna, Spain. \href{mailto:alu0101430720@ull.edu.es}alu0101430720@ull.edu.es}

\begin{abstract}
We derive a unidirectional asymptotic model for one-dimensional blood flow in viscoelastic arteries. We prove local well-posedness of strong solutions in Sobolev spaces for general parameters and mean-zero periodic data. In the purely elastic BBM regime we further establish global existence and exponential decay for sufficiently small initial data. We also present a numerical study of the reduced model, including comparisons across different viscoelastic and amplitude regimes, and discuss the observed dynamics in connection with the continuation criterion.
\end{abstract}

\maketitle
\tableofcontents

\section{Introduction}
The modeling of the human arterial system dates back to the works of Euler, who formulated the partial differential equations that describe the conservation of mass and momentum in inviscid flow. Early mathematical descriptions of arterial blood flow date back to Euler, while Young first identified its wave-like nature and derived an expression for the propagation velocity by analogy with wave motion in elastic tubes \cite{father_hemo, hemomath}. \medskip

The one-dimensional modeling of blood flow has been established as a valuable technique for studying circulatory dynamics in arteries and veins. This approach simplifies the vascular system into a single dimension, allowing for the analysis of blood flow along arterial or venous segments with adequate accuracy and reasonable computational cost. Essentially, these models describe the relationship between pressure, flow, and the cross-sectional area of blood vessels over time and distance. More precisely, the cross-sectional area of the vessel $A(x,t)$, the flow rate $Q(x,t)$ and the average internal pressure $p(x,t)$ over the cross section satisfy the conservation of mass and momentum balance equations. For mathematical convenience, and in order to avoid boundary effects at the level of the reduced asymptotic model, we work throughout on the periodic domain $\mathbb{T}=(\mathbb{R}/2\pi\mathbb{Z})$. Thus,
\begin{subequations}\label{sistema:1}
    \begin{align}
        A_t + Q_x &= 0, \quad x\in \mathbb{T}, t>0,\\
        Q_t + \alpha\left(\frac{Q^2}{A}\right)_x &= -\frac{A}{\rho}p_x + \frac{f}{\rho},  \quad x\in \mathbb{T}, t>0.
    \end{align}
\end{subequations}

Here $\rho$ is the fluid density assumed to be constant, $\alpha$ denotes a Coriolis coefficient and $f$ is a friction term, cf. \cite{one_dimen}.  To close the system, we must specify a constitutive law between the internal pressure $p$ and the cross-sectional area of the vessel $A$. In this work, we will consider the so called Kelvin-Voigt relation where the pressure is given by
\begin{equation}\label{Relacion_p_visco}
    p(A,x)= p_{\text{ext}} + \frac{\beta}{A_0(x)}\left(\sqrt{A(x,t)} - \sqrt{A_0(x)}\right) + \frac{\nu}{A_0(x)}(\sqrt{A(x,t)})_t,
\end{equation}
where
\begin{equation}\label{Beta}
    \beta(x) = \frac{\sqrt{\pi} h_0(x) E(x)}{1-\sigma^2}.
\end{equation}
Here, \( p_{\text{ext}} \) represents the external constant pressure, \( h_0(x)\) is the arterial wall thickness, and \( A_0(x) \) is the cross-sectional area in the equilibrium state \((p, u) = (p_{\text{ext}}, 0)\). The constant \( \nu \) is a viscoelastic coefficient that depends on the artery's thickness, \( E(x) \) denotes Young’s modulus, and \( \sigma \) is Poisson’s ratio. Equation \eqref{Relacion_p_visco} incorporates a viscoelastic correction to the pressure, with the coefficient \( \beta \) defined in \eqref{Beta}, which accounts for the mechanical properties of the arterial wall.  In order to make the model more tractable is customary to choose $A_{0}$ and $\beta$ as positive constants independent of the $x$-variable.\medskip
\\

An alternative formulation of the system of governing equations \eqref{sistema:1} can be obtained for the triple $(A,u,p)$ where $u(x,t)$ denotes the blood velocity. Indeed, writing $Q=Au$ and taking $\alpha=1$ we find the alternative system 
\begin{subequations}\label{sistema:2}
    \begin{align}
        A_t + (Au)_x&= 0,\label{eq1:sistema2}\\
        u_t + uu_x &= -\frac{p_x }{\rho}+ \frac{f}{\rho A}.\label{eq2:sistema2}
    \end{align}
\end{subequations}
together with the same constitutive law
\begin{equation}\label{Relacion_p_visco:2}
    p= p_{\text{ext}} + \frac{\beta}{A_0}\left(\sqrt{A} - \sqrt{A_0}\right) + \frac{\nu}{A_0}(\sqrt{A})_t,
\end{equation}
Following the work \cite{benchmark}, the friction term $f=f(u)$ is defined as a Poiseuille parabolic velocity profile $f(u)=\kappa u$ where $\kappa\geq 0$ is related to the blood viscosity.

\subsection*{Previous works and results}
System \eqref{sistema:1}  together with the relation \eqref{Relacion_p_visco} was introduced in \cite{Alastruey2011} and its well-posedness has been studied later in \cite{Canic2006, viscoelastico }. In particular, in \cite{viscoelastico} the authors studied the existence and uniqueness of maximal solutions with suitable nonlinear Robin boundary conditions. Additionally, it has been utilized in \cite{Lal2017,Caiazzo2017 } for hemodynamic parameter estimation and in \cite{Montecinos2014,Wang2015,Wang2016} for the development and analysis of numerical schemes, particularly those accounting for the viscoelastic correction term. For more general constitutive laws including for instance second time derivatives we refer the interested reader to \cite{Armentano,Canic2006}. \medskip

When the viscoelastic coefficient \( \nu \) is set to zero, system \eqref{sistema:1}-\eqref{Relacion_p_visco} can be written as an hyperbolic quasilinear system and several well-posedness results via classical hyperbolic techniques are available in the literature \cite{Canic2003, Montecinos2014}. However, it seems that from the numerical point of view \cite{Muller2016} the viscoelastic term plays a significant role when comparing numerical models with in vivo data \cite{Lal2017}. It has been shown that incorporating viscoelastic wall models yields more physiologically accurate predictions than purely elastic models. Indeed, one-dimensional elastic models tend to overestimate both blood pressure and vessel deformation, as highlighted in \cite{benchmark,one_dimen, gravity}. \medskip

\subsection{Contributions and main results}\label{subsec:contrib}

The purpose of this paper is twofold. First, we derive a unidirectional asymptotic equation associated with the
one-dimensional blood flow model \eqref{sistema:1}--\eqref{Relacion_p_visco}. The derivation relies on a multi-scale
expansion in the small-amplitude/long-wave parameter $0<\varepsilon\ll1$ (cf. \cite{AODGB24,AGBSW19,GBS21}),
which reduces the full system to a hierarchy of linear problems that can be closed at a prescribed order of accuracy.
More precisely, we write
\[
A=\overline A_0+\varepsilon h = 1+\varepsilon h,
\qquad
u=\overline u_0+\varepsilon U=\varepsilon U,
\]
and introduce the formal expansions
\[
h(x,t)=\sum_{\ell=0}^{\infty}\varepsilon^\ell h^{(\ell)}(x,t),
\qquad
U(x,t)=\sum_{\ell=0}^{\infty}\varepsilon^\ell U^{(\ell)}(x,t).
\]
Truncating the expansion at order $\mathcal O(\varepsilon^2)$ yields the following unidirectional model for
\eqref{sistema:1}--\eqref{Relacion_p_visco}:
\begin{align}\label{asymptotic:model:contr}
	f_{t}&=\mathcal{M}\mathcal{P}\bigg[ - \frac{1}{\varepsilon}\Big(1-\frac{\beta}{2}\Big)f_{xx}
	+ \frac{1}{\varepsilon}\kappa f_{x}
	- \frac{1}{\varepsilon}\frac{\nu}{2}f_{xxx}
	+ \Big(2+\frac{\beta}{4}\Big)(f f_{x})_{x}    \nonumber \\
	&\hspace{4.7cm}+ \frac{1}{\varepsilon}\frac{\nu}{4}f_{x}f_{xx}
	-\frac{\nu}{4}ff_{xxx}
	-2\kappa ff_{x}\bigg],
\end{align}
for the asymptotic unknown $f(x,t):=h^{(0)}(x,t)+\varepsilon h^{(1)}(x,t)$.
The nonlocal operators in \eqref{asymptotic:model:contr} are Fourier multipliers defined by
\begin{equation}\label{opP}
	\mathcal{P}:=\Big(\kappa-\frac{\nu}{2}\partial_{xx}\Big)^{-1},
	\qquad
	\mathcal{M}:=\left(\Id-4\Big(\frac{\partial_x}{\kappa-\frac{\nu}{2}\partial_{xx}}\Big)^{2}\right)^{-1}
	\left(\Id+\frac{2\partial_x}{\kappa-\frac{\nu}{2}\partial_{xx}}\right),
\end{equation}
with symbols
\begin{equation}\label{multiplierP-multiplierM}
	\widehat{\mathcal{P}f}(k)=\frac{1}{\kappa+\frac{\nu}{2}|k|^{2}}\widehat{f}(k),
	\quad
	\widehat{\mathcal{M}f}(k)= \mathsf{m}(k)\widehat{f}(k),
	\quad
	\mathsf{m}(k)=\left(1 + 4 \frac{|k|^2}{\left(\kappa + \frac{\nu}{2} |k|^2\right)^2} \right)^{-1}
	\left(1 +  \frac{2 i k}{\kappa + \frac{\nu}{2} |k|^2}\right).
\end{equation}
The second goal of the paper is to study analytical properties of the reduced model \eqref{asymptotic:model:contr}.
Our main results can be summarized as follows:
\begin{itemize}
	\item \emph{Local well-posedness}. For $s>\frac52$ and mean-zero data $f_0\in H^s(\TT)$, we prove existence and uniqueness
	of a strong solution $f\in C([0,T];H^s(\TT))$ on a time interval $[0,T]$ depending only on $\|f_0\|_{H^s}$ and the physical
	parameters; see Theorem~\ref{th:existence:strong} in Section~\ref{sec:wp}.
	\item \emph{Global small-data theory in the BBM regime}. In the purely elastic case $\nu=0$ (for which
	\eqref{asymptotic:model:contr} can be written in a BBM-type local form), and assuming $\beta>-2$, we establish global
	existence and decay in $H^2(\TT)$ for sufficiently small mean-zero initial data; see
	Theorem~\ref{th:bbm-global} in Section~\ref{sec:bbm-global}.
\item \emph{Numerical simulations}. Finally, we present a numerical study of the asymptotic model, including experiments in the viscoelastic regime ($\nu>0$) for different amplitudes and parameter values, as well as simulations in the purely elastic case ($\nu=0$). These computations illustrate the qualitative behavior of the solutions across several regimes and are interpreted in light of the continuation criterion; see Section~\ref{sec:numerical}.
\end{itemize}
\medskip
From a modeling viewpoint, equation \eqref{asymptotic:model:contr} describes the slow modulation (on the long time scale
$\tau=\varepsilon t$) of small-amplitude, long-wave disturbances propagating predominantly in one direction along the vessel,
i.e.\ a single traveling-wave branch selected by the far-field change of variables $\xi=x-t$. In this reduced regime,
$\beta$ encodes the effective wall elasticity (and thus the characteristic wave speed), while $\kappa$ accounts for viscous
damping due to frictional losses (e.g.\ Poiseuille-type resistance) and $\nu$ introduces a viscoelastic correction that
regularizes the dynamics through higher-order dispersive/dissipative effects. Consequently, the competition between
elastic propagation, damping, and viscoelastic regularization is captured at the level of \eqref{asymptotic:model:contr},
providing a tractable one-dimensional model for wave dynamics in compliant arteries.

\subsection{Notation and preliminaries}
Let us next introduce the notation that will be used throughout the rest of the paper. 
	For  $1\leq p<\infty$ we denote by  $L^p(\TT;\RR)$ the standard Lebesgue space of measurable $p$-integrable $\RR$-valued functions with domain $\TT=(\RR/2\pi\mathbb Z)$ and by  $L^\infty(\TT;\RR)$ the space of essentially bounded functions. Particularly, $L^2(\TT;\RR)$ is equipped with the inner product 
	$
	\langle f,g\rangle_{L^2}=\int_{\TT}f\cdot\overline{g}\, dx,
	$
	where $\overline{g}$ denotes the complex conjugate of $g$.  \medskip
	
	The Fourier
	transform and inverse Fourier transform of $f(x)\in L^2(\TT;\RR)$ are defined by
	$\widehat{f}(k)=\int_{\TT}f(x){\rm e}^{-{\rm i}x k}\, dx$ and 
	$f(x)=\frac{1}{2\pi}\sum_{k\in\mathbb Z}\widehat{f}(k){\rm e}^{{\rm i}x k}$, respectively. Recalling that for any $s\in\RR$,  
	$\widehat{D^{s} f}(k)=(1+|k|^2)^{s/2}\widehat{f}(k)$, we define the Sobolev space $H^s$ on $\TT$ with values in $\RR$ as
	\begin{align*}
	H^s(\TT;\RR):=\left\{f\in L^2(\TT;\RR):\|f\|_{H^s(\TT;\RR)}^2
	=\sum_{k\in\mathbb Z}|\widehat{D^sf}(k)|^2<+\infty\right\}.
	\end{align*}
	Recall that this norm is equivalent to 
	\[ \norm{f}^{2}_{H^{s}(\TT)}= \norm{f}_{L^{2}(\TT)}^{2}+\norm{f}_{\dot{H}^{s}(\TT)}^{2},\]	
	where $\norm{f}_{\dot{H}^{s}(\TT)}^{2}=\norm{\Lambda^{s}f}_{L^{2}(\TT)}^{2}$ and $\Lambda^{s}$ is defined as the homogeneous multiplier of $D^{s}$, namely, $\widehat{\Lambda^{s} f}(k)=|k|^{s}\widehat{f}(k)$.  Throughout the paper $C = C(\cdot)$ will denote a positive constant that may depend on fixed parameters  and  $x \lesssim y$ ($x \gtrsim y$) means that $x\le C y$ ($x\ge C y$) holds for some $C$.

\subsubsection*{Calculus estimates, operators and symbols}
	Next, let us recall the product estimate in Sobolev spaces, the so called Kato-Ponce commutator  \cite{Kato-Ponce-1988-CPAM, Kenig-Ponce-Vega-1991-JAMS}
	\begin{equation}\label{KPcom}
	\|\left[\Lambda^s,f\right]g\|_{L^p(\TT)}\leq C_{s,p}(\|\pa_{x} f\|_{L^{p_1}(\TT)}\|\Lambda^{s-1}g\|_{L^{p_2}(\TT)}+\|\Lambda^sf\|_{L^{p_3}(\TT)}\|g\|_{L^{p_4}(\TT)}),
	\end{equation}
	with $p,p_i\in(1,\infty)$ with $i=1,\ldots,4$ and $\frac{1}{p}=\frac{1}{p_1}+\frac{1}{p_2}=\frac{1}{p_3}+\frac{1}{p_4}$. We will also make use of the Sobolev embedding and algebra property
		\begin{align}
		\norm{f}_{L^{\infty}(\TT)}&\leq C_{s} \norm{f}_{\dot{H}^{s}(\TT)}, \label{Sob:e} \\
		\norm{fg}_{\dot{H}^{s}(\TT)}&\leq C_{s} \norm{f}_{\dot{H}^{s}(\TT)}\norm{g}_{\dot{H}^{s}(\TT)}, \label{Alg:pro}
		\end{align}
	for $f,g $ a zero mean functions and $s>1/2$. \medskip
	
In the following lemma we provide some identities and estimates for the operators \eqref{opP}-\eqref{multiplierP-multiplierM}:

\begin{lemma}\label{lemma:op} 
Assume $\nu>0$ and $\kappa>0$. Let $\mathcal{P}$ and $\mathcal{M}$ be the differential operators given in \eqref{opP} and \eqref{multiplierP-multiplierM}. Then,
\begin{enumerate}
\item $\mathcal{P}$ is a smoothing operator of degree $-2$ such that for $s\in\mathbb{R}$ and $f\in H^{s}(\mathbb{T})$ 
\begin{equation}\label{estimate:P:smooth}
\norm{\mathcal{P}f}_{H^{s+2}(\mathbb{T})}\lesssim  \norm{f}_{H^{s}(\mathbb{T})}.
\end{equation}
Furthermore, we have the identity
\begin{equation}\label{twoder:helm}
\partial_{xx}\mathcal{P}=-\frac{2}{\nu}\mathsf{Id}+\frac{2\kappa }{\nu}\mathcal{P}.
\end{equation}

\item $\mathcal{M}f=\left(\mathsf{Id}+\mathcal{S}\right)f$,
where $\mathcal{S}$ is a smoothing operator of degree $-1$. In particular, for $s\in\mathbb{R}$ and $f\in H^{s}(\mathbb{T})$ we have that
\begin{equation}\label{bound:smooth}
\norm{\mathcal{S}f}_{H^{s+1}(\mathbb{T})}\lesssim  \norm{f}_{H^{s}(\mathbb{T})}.
\end{equation}

\end{enumerate} 
\end{lemma}

\begin{proof}[Proof of Lemma \ref{lemma:op}]
Throughout the proof we use the Fourier convention on $\TT$ given by
$\widehat{\partial_x f}(k)=ik\widehat f(k)$ and $\widehat{\partial_{xx} f}(k)=-k^2\widehat f(k)$.
We also set
\[
a(k):=\kappa+\frac{\nu}{2}|k|^2>0,\qquad k\in\ZZ .
\]

\medskip\noindent
\underline{\textit{Step 1: proof of \eqref{estimate:P:smooth} and \eqref{twoder:helm}.}}
Using \eqref{multiplierP-multiplierM} we find that
\begin{align*}
\norm{\mathcal{P}f}^{2}_{H^{s+2}(\mathbb{T})}
=\sum_{k\in\mathbb{Z}} \left(1+|k|^{2}\right)^{s+2}\frac{1}{a(k)^{2}}|\widehat{f}(k)|^{2}&=\sum_{k\in\mathbb{Z}} \left(1+|k|^{2}\right)^{s}\,
\frac{\left(1+|k|^{2}\right)^{2}}{a(k)^{2}}\,|\widehat{f}(k)|^{2}\\
&\le \Big(\sup_{k\in\mathbb{Z}}\frac{\left(1+|k|^{2}\right)^{2}}{a(k)^{2}}\Big)
\sum_{k\in\mathbb{Z}}\left(1+|k|^{2}\right)^{s} |\widehat{f}(k)|^{2}.
\end{align*}
It remains to show that the supremum is finite. Since $a(k)\ge \kappa$ for all $k$ and
$a(k)\ge \frac{\nu}{2}|k|^{2}$ for $|k|\ge 1$, we get
\[
\sup_{|k|\le 1}\frac{(1+k^{2})^{2}}{a(k)^2}\le \frac{(1+1)^2}{\kappa^2}=\frac{4}{\kappa^2},
\qquad
\sup_{|k|\ge 1}\frac{(1+k^{2})^{2}}{a(k)^2}
\le \sup_{|k|\ge 1}\frac{(2k^{2})^{2}}{\left(\frac{\nu}{2}k^{2}\right)^2}
=\frac{16}{\nu^{2}}.
\]
Therefore $\sup_{k\in\ZZ}\frac{(1+k^{2})^{2}}{a(k)^2}\lesssim 1$ (with an implicit constant depending only on
$\kappa,\nu$), and bound \eqref{estimate:P:smooth} follows. For \eqref{twoder:helm}, we compute the symbol of $\partial_{xx}\mathcal{P}$:
\[
\widehat{\partial_{xx}\mathcal{P} f}(k)=\frac{-k^{2}}{a(k)}\widehat{f}(k)=\frac{2}{\nu}\Big(\frac{\kappa}{a(k)}-1\Big)\widehat{f}(k)
=\left(-\frac{2}{\nu}+\frac{2\kappa}{\nu}\frac{1}{a(k)}\right)\widehat{f}(k).
\]
Taking inverse Fourier transform yields the operator identity
\[
\partial_{xx}\mathcal{P}
=-\frac{2}{\nu}\mathsf{Id}+\frac{2\kappa}{\nu}\mathcal{P}.
\]
\medskip\noindent
\underline{\textit{Step 2: decomposition $\mathcal{M}=\Id+\mathcal{S}$ and proof of \eqref{bound:smooth}.}}
From \eqref{multiplierP-multiplierM} we have
\[
\mathsf{m}(k)=\left(1 + 4 \frac{|k|^2}{a(k)^2} \right)^{-1}
\left(1 +  \frac{2 i k}{a(k)}\right)
=\frac{a(k)^2}{a(k)^2+4k^2}\cdot \frac{a(k)+2ik}{a(k)}
=\frac{a(k)\big(a(k)+2ik\big)}{a(k)^2+4k^2}.
\]
Define $\mathsf{s}(k)$ by
\begin{equation}\label{derivada:algo}
\mathsf{m}(k)=1-\mathsf{s}(k),
\qquad\text{that is}\qquad
\mathsf{s}(k)=1-\frac{a(k)\big(a(k)+2ik\big)}{a(k)^2+4k^2}.
\end{equation}
A direct computation gives
\[
\mathsf{s}(k)=\frac{a(k)^2+4k^2-a(k)\big(a(k)+2ik\big)}{a(k)^2+4k^2}
=\frac{4k^2-2ik\,a(k)}{a(k)^2+4k^2}
=-\,\frac{2ik\big(a(k)-2ik\big)}{a(k)^2+4k^2}.
\]
In particular,
\[
|\mathsf{s}(k)|^{2}
=\frac{4k^{2}\,|a(k)-2ik|^{2}}{(a(k)^{2}+4k^{2})^{2}}
=\frac{4k^{2}\,(a(k)^{2}+4k^{2})}{(a(k)^{2}+4k^{2})^{2}}
=\frac{4k^{2}}{a(k)^{2}+4k^{2}}.
\]
This shows that $\mathsf{s}(k)=\mathcal{O}(1/|k|)$ as $|k|\to\infty$, hence $\mathcal{S}$ is a smoothing
operator of degree $-1$. Moreover, it provides the uniform bound needed below:
\[
(1+k^{2})|\mathsf{s}(k)|^{2}=\frac{4(1+k^{2})k^{2}}{a(k)^{2}+4k^{2}}
\le \max\Big\{\frac{8}{\kappa^{2}},\frac{32}{\nu^{2}}\Big\}
\qquad\text{for all }k\in\ZZ,
\]
where we used that $a(k)\ge \kappa$ for $|k|\le 1$ and $a(k)\ge \frac{\nu}{2}k^{2}$ for $|k|\ge 1$. \medskip

Now let $\mathcal{S}$ be the Fourier multiplier operator with symbol $\mathsf{s}(k)$, i.e.
\[
\widehat{\mathcal{S}f}(k)=-\mathsf{s}(k)\widehat f(k).
\]
Then \eqref{derivada:algo} implies $\widehat{\mathcal{M}f}(k)=\big(1+\mathsf{s}(k)\big)\widehat f(k)$, i.e. $\mathcal{M}=\Id+\mathcal{S}.$ Finally, to prove \eqref{bound:smooth}, we compute
\begin{align*}
\norm{\mathcal{S}f}^{2}_{H^{s+1}(\mathbb{T})}
&=\sum_{k\in\mathbb{Z}}  \left(1+|k|^{2}\right)^{s+1} |\mathsf{s}(k)|^{2}|\widehat{f}(k)|^{2}\\
&\le \Big(\sup_{k\in\mathbb{Z}} (1+|k|^{2}) |\mathsf{s}(k)|^{2}\Big)
\sum_{k\in\mathbb{Z}}  \left(1+|k|^{2}\right)^{s} |\widehat{f}(k)|^{2}\\
&\lesssim \norm{f}_{H^{s}(\mathbb{T})}^{2},
\end{align*}
where the last inequality follows from the uniform bound on
$\sup_k (1+k^2)|\mathsf{s}(k)|^2$ proved above. This yields \eqref{bound:smooth} and completes the proof.
\end{proof}

\begin{remark}[On the cases $\nu=0$ and $\kappa=0$ in Lemma~\ref{lemma:op}]\label{rem:nu0-kappa0}
	Lemma~\ref{lemma:op} is stated under $\nu>0$ and $\kappa>0$ so that the operator
	$\mathcal P=(\kappa-\frac{\nu}{2}\partial_{xx})^{-1}$ is an everywhere defined Fourier multiplier on $H^s(\TT)$
	(in particular on the zero mode) and the identity \eqref{twoder:helm} makes sense.
	
	\smallskip
	\noindent\emph{(i) The purely elastic case $\nu=0$.}
	When $\nu=0$ we have $\mathcal P=\kappa^{-1}\Id$ and hence $\mathcal P$ no longer provides two derivatives of smoothing;
	instead it is a bounded zeroth-order multiplier. In particular,
	\[
	\|\mathcal P f\|_{H^s(\TT)}\le \kappa^{-1}\|f\|_{H^s(\TT)},\qquad s\in\RR,
	\]
	and \eqref{twoder:helm} is not available. The operator $\mathcal M$ remains well-defined for $\kappa>0$ and
	\eqref{multiplierP-multiplierM} reduces to
	\[
	\mathsf{m}(k)=\Big(1+4\frac{|k|^2}{\kappa^2}\Big)^{-1}\Big(1+\frac{2ik}{\kappa}\Big),
	\]
	so that $\mathcal M=\Id+\mathcal S$ still holds with a smoothing remainder $\mathcal S$ of degree $-1$; in particular,
	\eqref{bound:smooth} remains valid (with constants depending on $\kappa$).
	
	\smallskip
	\noindent\emph{(ii) Vanishing friction $\kappa=0$.}
	If $\kappa=0$ and $\nu>0$, then $\mathcal P=((\nu/2)(-\partial_{xx}))^{-1}$ is not defined on the zero Fourier mode.
	However, on the subspace of mean-zero functions (i.e.\ $\widehat f(0)=0$) one can still define $\mathcal P$ by
	\[
	\widehat{\mathcal P f}(k)=\frac{2}{\nu|k|^2}\widehat f(k),\qquad k\neq 0,\qquad \widehat{\mathcal P f}(0)=0,
	\]
	which yields the smoothing estimate \eqref{estimate:P:smooth} on mean-zero data.
	In this case \eqref{twoder:helm} simplifies to $\partial_{xx}\mathcal P=-\frac{2}{\nu}\Id$ on mean-zero functions.
	The multiplier $\mathsf{m}(k)$ in \eqref{multiplierP-multiplierM} is also well-defined for $k\neq 0$ and extends to
	the mean-zero subspace; consequently, the decomposition $\mathcal M=\Id+\mathcal S$ and the bound \eqref{bound:smooth}
	remain valid on mean-zero functions. From a modeling viewpoint, setting $\kappa=0$ suppresses the viscous (e.g.\ the Poiseuille-type damping), and is therefore not physiologically meaningful for blood flow in arteries except as a purely idealized limit. For this reason, we do not treat the case $\kappa=0$ in the present work.
\end{remark}

\subsection{Plan of the paper}
In Section~\ref{sec:derivation} we present the asymptotic derivation of a unidirectional model from the blood flow system
\eqref{sistema:2} using a multi-scale expansion. Section~\ref{sec:wp} is devoted to the local well-posedness of the resulting
unidirectional equation \eqref{asymptotic:model:contr}, obtained via a priori energy estimates and a standard mollification
procedure. In Section~\ref{sec:bbm-global} we consider the BBM regime ($\nu=0$) and prove global existence together with decay
for sufficiently small initial data. We conclude with numerical simulations for the full viscoelastic model ($\nu>0$), which
suggest that the local strong solutions constructed in Section~\ref{sec:wp} may develop a finite-time singularity.

\section{Derivation of the asymptotic blood flow models}\label{sec:derivation}
In this section, we provide the derivation of the asymptotic models \eqref{asymptotic:model:contr} by means of a multi-scale expansion. In order to obtain the first asymptotic model we recall that the system is given by
\begin{subequations}\label{system1:der}
    \begin{align}
        A_t + (Au)_x &= 0,\\
        u_t + u u_x &= -p_x - \frac{\kappa u}{A},
    \end{align}
\end{subequations}
\noindent together with the constitutive pressure law
\begin{equation*}
    p = p_{\text{ext}} + \frac{\beta}{A_0}\left(\sqrt{A} - \sqrt{A_0}\right) + \frac{\nu}{A_0}(\sqrt{A})_t.
\end{equation*}
For the sake of simplicity, we recall that we will take $A_0 = 1$. Thus, combining plugging the constitutive law into \eqref{system1:der}, the system is given by
\begin{subequations}\label{sistema_2.3}
    \begin{align}
        A_t + (Au)_x &=0,\\
   A\left[ u_t +uu_x + \beta \partial_x \left(\sqrt{A}\right) +  \nu \partial_{xt} \left(\sqrt{A}\right)\right]&=-{\kappa u}.
    \end{align}
\end{subequations}
Next, we linearize system \eqref{sistema_2.3} around the trivial solution \[\overline{u}_0=0,\ \overline{A}_0=1.\] More precisely, we look for perturbed solutions of the form
\begin{equation}\label{soluciones_cerca_triviales}
    A = \overline{A}_0 + \varepsilon h = 1+\varepsilon  h,\hspace{5mm} u= \overline{u}_0 + \varepsilon U = \varepsilon U,\hspace{5mm} 0<\varepsilon\ll 1.
\end{equation}

Substituting \eqref{soluciones_cerca_triviales} into system \eqref{sistema_2.3}, we observe that
\begin{subequations}
    \begin{align*}
        (1 + \varepsilon h)_t + ((1 + \varepsilon h)\varepsilon U)_x &= 0,\\
        (1 + \varepsilon h)\left[(\varepsilon U)_t + \varepsilon U (\varepsilon U)_x + \beta\partial_x\left(\sqrt{1 + \varepsilon h}\right)+\nu \partial_{xt}\left(\sqrt{1+ \varepsilon h }\right)\right]&=-\kappa \varepsilon U.
    \end{align*}
\end{subequations}

Using the Taylor expansion
\begin{equation*}
	\sqrt{1 + \varepsilon h}
	= 1+ \frac{1}{2} \varepsilon  h -\frac{1}{8} \varepsilon^2h^2 +\mathcal{O}( \varepsilon^3),
\end{equation*}
we expand the derivatives in $\varepsilon$. Notice that every term in the bracket in the second equation above is of order
$\mathcal{O}(\varepsilon)$; hence we may divide the second equation by $\varepsilon$.
Keeping all contributions up to order $\varepsilon$ (and discarding $\mathcal{O}(\varepsilon^2)$ remainders), we obtain
\begin{subequations}
	\begin{align*}
		h_t +  U_x + \varepsilon(hU)_x &=0,\\
		(1+\varepsilon h)U_t +\varepsilon UU_x + \frac{\beta}{2} h_x + \frac{\beta}{4}\varepsilon hh_x
		+  \nu \left(\frac{1}{2} h_{xt} - \frac{1}{4}\varepsilon h_th_x + \frac{1}{4}\varepsilon hh_{xt}\right)+\kappa U&=0.
	\end{align*}
\end{subequations}

The main idea is to construct an asymptotic expansion of solutions to the previous system in the small parameter
$0<\varepsilon\ll1$. We therefore seek $h$ and $U$ in the form of the formal series
\begin{equation}\label{eq:ansatz-series}
	h(x,t)= \sum_{\ell=0}^{\infty} \varepsilon^\ell h^{(\ell)}(x,t),
	\qquad
	U(x,t) = \sum_{\ell=0}^{\infty} \varepsilon^{\ell}U^{(\ell)}(x,t).
\end{equation}
Substituting \eqref{eq:ansatz-series} into the system and identifying coefficients of equal powers of $\varepsilon$,
we obtain a cascade of linear forced problems: for each $\ell\ge 0$, the pair $(h^{(\ell)},U^{(\ell)})$ solves
\begin{subequations}\label{eq:cascade}
	\begin{align}
		h_t^{(\ell)} + U_x^{(\ell)}
		+ \sum_{j=0}^{\ell-1}\big(h^{(j)}U^{(\ell-1-j)}\big)_x
		&=0,\label{eq:cascade-mass}\\[0.2cm]
		U_t^{(\ell)}+  \frac{\beta}{2} h_x^{(\ell)}+\frac{\nu}{2} h_{xt}^{(\ell)}+\kappa U^{(\ell)}
		+ \sum_{j=0}^{\ell-1}\bigg(
		h^{(j)} U_t^{(\ell-1-j)}+ U^{(j)}U_x^{(\ell-1-j)}
		+\frac{\beta}{4}  h^{(j)}h_x^{(\ell-1-j)}
		&\nonumber\\
		\hspace{4.2cm}
		-\frac{\nu}{4}h_t^{(j)}h_x^{(\ell-1-j)}
		+\frac{\nu}{4}h^{(j)}h_{xt}^{(\ell-1-j)}
		\bigg)
		&=0.\label{eq:cascade-mom}
	\end{align}
\end{subequations}
Here, by convention the sums are empty when $\ell=0$, so that \eqref{eq:cascade} reduces to the linear homogeneous
system for $(h^{(0)},U^{(0)})$, while for $\ell\ge1$ the right-hand sides are determined by the profiles computed at
lower orders. This cascade of equations can be solved recursively. In particular, the first term corresponding to $\ell=0$ is given by the system
\begin{subequations}\label{sistema_2.11}
	\begin{align}
		h^{(0)}_t + U^{(0)}_x  &= 0, \label{sistema_2.11a} \\
		U^{(0)}_t + \frac{\beta}{2} h^{(0)}_x+\frac{\nu}{2}h^{(0)}_{xt} + \kappa U^{(0)}&=0. \label{sistema_2.11b}
	\end{align}
\end{subequations}

Thus, by differentiating \eqref{sistema_2.11a} with respect to $t$ and \eqref{sistema_2.11b} with respect to $x$, we obtain
\[
h^{(0)}_{tt} = -U^{(0)}_{xt}, \qquad
U^{(0)}_{xt} + \frac{\beta}{2} h^{(0)}_{xx}+\frac{\nu}{2}h^{(0)}_{xxt} + \kappa U^{(0)}_{x}=0.
\]
Combining both identities yields
\begin{equation}
	h^{(0)}_{tt} =  \frac{\beta}{2}  h^{(0)}_{xx}+\frac{\nu}{2}h^{(0)}_{xxt} + \kappa U^{(0)}_x.
\end{equation}
Since \eqref{sistema_2.11a} implies $U_x^{(0)}=-h_t^{(0)}$, we conclude that
\begin{equation}
	h^{(0)}_{tt} =  \frac{\beta}{2}  h^{(0)}_{xx}+\frac{\nu}{2}h^{(0)}_{xxt} - \kappa h^{(0)}_t,
\end{equation}
or equivalently,
\begin{equation}\label{eq2.15}
	\left(\partial_{tt} - \frac{\beta}{2}\partial_{xx} + \kappa\partial_t-\frac{\nu}{2}\partial_{xxt}\right)h^{(0)}=0.
\end{equation}
The next term in the cascade, corresponding to $\ell=1$, is given by
\begin{subequations}\label{sistema_2.16}
	\begin{align}
		h^{(1)}_t + U^{(1)}_x + \left(h^{(0)}U^{(0)}\right)_x
		&= 0,\label{2.16a}\\[0.15cm]
		U_t^{(1)}+  \frac{\beta}{2} h_x^{(1)}+\frac{\nu}{2} h_{xt}^{(1)}+\kappa U^{(1)}
		+\bigg( h^{(0)} U_t^{(0)}+ U^{(0)}U_x^{(0)}+\frac{\beta}{4}  h^{(0)}h_x^{(0)}
		&\nonumber\\
		\hspace{3.2cm}
		-\frac{\nu}{4}h_t^{(0)}h_x^{(0)}+\frac{\nu}{4}h^{(0)}h_{xt}^{(0)}\bigg)
		&=0.\label{2.16b}
	\end{align}
\end{subequations}
Differentiating \eqref{2.16a} with respect to $t$ and \eqref{2.16b} with respect to $x$, and eliminating $U^{(1)}_{xt}$, we obtain
\begin{equation}\label{eq2.17}
	h^{(1)}_{tt}-\frac{\beta}{2}h^{(1)}_{xx}-\frac{\nu}{2}h^{(1)}_{xxt}-\kappa U^{(1)}_{x}
	=-\left(h^{(0)}U^{(0)}\right)_{xt}
	+\bigg(h^{(0)} U_t^{(0)}+ U^{(0)}U_x^{(0)}
	+\frac{\beta}{4}  h^{(0)}h_x^{(0)}
	-\frac{\nu}{4}h_t^{(0)}h_x^{(0)}+\frac{\nu}{4}h^{(0)}h_{xt}^{(0)}\bigg)_{x}.
\end{equation}
Using \eqref{2.16a}, we have
\[
U^{(1)}_x=-h^{(1)}_t-\big(h^{(0)}U^{(0)}\big)_x.
\]
Substituting this identity into \eqref{eq2.17} and regrouping the terms, we obtain
\begin{equation}\label{eq:LF}
	\left(\partial_{tt} - \frac{\beta}{2}\partial_{xx} + \kappa\partial_t-\frac{\nu}{2}\partial_{xxt}\right)h^{(1)}
	=\mathcal{F}(h^{(0)}, U^{(0)}),
\end{equation}
where the forcing term is given by
\begin{equation}\label{funcionF}
	\mathcal{F}(h^{(0)}, U^{(0)}):=\bigg(-\kappa h^{(0)}U^{(0)}-h_{t}^{(0)}U^{(0)}+U^{(0)}U_x^{(0)}
	+\frac{\beta}{4}  h^{(0)}h_x^{(0)}
	-\frac{\nu}{4}h_t^{(0)}h_x^{(0)}+\frac{\nu}{4}h^{(0)}h_{xt}^{(0)}\bigg)_{x}.
\end{equation}

Observing now that from \eqref{sistema_2.11a} we have $U_x^{(0)}=-h_t^{(0)}$, we can recover $U^{(0)}$
(up to its spatial mean) by applying the periodic inverse derivative. Imposing the normalization $\widehat{U^{(0)}}(0,0)=0$, and observing that the zero-mean condition $\widehat{U^{(0)}}(0,t)=0$ is preserved by the evolution equation \eqref{sistema_2.11b}, we define
\begin{equation}\label{eq2.22}
	U^{(0)}(x,t):=-\partial_x^{-1}h_t^{(0)}(x,t),
\end{equation}
where $\partial_x^{-1}$ denotes the Fourier multiplier given by
$\widehat{\partial_x^{-1}g}(k)=\frac{1}{ik}\widehat g(k)$ for $k\neq 0$ and $\widehat{\partial_x^{-1}g}(0)=0$. Therefore, using  \eqref{sistema_2.11a} together with \eqref{eq2.22}, we can rewrite $\mathcal{F}$ only in terms of $h^{(0)}$, namely
\begin{equation}\label{2.23-simplified}
\begin{aligned}
\mathcal{F}(h^{(0)})=
&\kappa h_x^{(0)}\,\partial_x^{-1}h_t^{(0)}
+\kappa h^{(0)} h_t^{(0)}
+2h_{tx}^{(0)}\,\partial_x^{-1}h_t^{(0)}
+2\big(h_t^{(0)}\big)^2\\
&+\frac{\beta}{4}\Big(\big(h_x^{(0)}\big)^2+h^{(0)}h_{xx}^{(0)}\Big)
-\frac{\nu}{4}h_t^{(0)}h_{xx}^{(0)}
+\frac{\nu}{4}h^{(0)}h_{xxt}^{(0)}.
\end{aligned}
\end{equation}
Thus, recalling \eqref{eq:LF} and \eqref{2.23-simplified}, we conclude that
	\begin{align}
		\left(\partial_{tt} - \frac{\beta}{2}\partial_{xx} + \kappa\partial_t-\frac{\nu}{2}\partial_{xxt}\right)h^{(1)}
		&=\mathcal{F}(h^{(0)}). \label{eq2.24a}
	\end{align}
Now, considering the new function
\[
f(x,t):=h^{(0)}(x,t) + \varepsilon h^{(1)}(x,t),
\]
using \eqref{eq2.15}--\eqref{eq2.24a} and neglecting terms of order $\mathcal{O}(\varepsilon^{2})$, we conclude that
$f$ satisfies
\begin{align}\label{eq2.29}
	\left(\partial_{tt} - \frac{\beta}{2}\partial_{xx} + \kappa\partial_t-\frac{\nu}{2}\partial_{xxt}\right)f
	&=\varepsilon\,\mathcal{F}(f)+\mathcal{O}(\varepsilon^2),
\end{align}
where $\mathcal{F}(f)$ is obtained from \eqref{2.23-simplified} by replacing $h^{(0)}$ with $f$. More precisely,
(with $\partial_x^{-1}$ understood in the periodic sense),
\begin{equation}\label{eq:Ff-simplified}
\mathcal{F}(f)=
\kappa f_x\,\partial_x^{-1}f_t
+\kappa f\, f_t
+2f_{tx}\,\partial_x^{-1}f_t
+2\big(f_t\big)^2
+\frac{\beta}{4}\Big(\big(f_x\big)^2+f f_{xx}\Big)
-\frac{\nu}{4}f_t f_{xx}
+\frac{\nu}{4}f\, f_{xxt}.
\end{equation}

Moreover, in order to derive a unidirectional version of \eqref{eq2.29}, we introduce the far-field variables
\[
\xi=x-t,\qquad \tau=\varepsilon t,
\]
and we write $f(x,t)=f(\xi,\tau)$. By the chain rule we have
\[
\partial_x=\partial_\xi,\qquad \partial_t=-\partial_\xi+\varepsilon\partial_\tau,
\]
and therefore
\[
f_x=f_\xi,\qquad f_{xx}=f_{\xi\xi},\qquad
f_t=-f_\xi+\varepsilon f_\tau,\qquad
f_{tt}=f_{\xi\xi}-2\varepsilon f_{\xi\tau}+\mathcal O(\varepsilon^2),
\qquad
f_{xxt}=-f_{\xi\xi\xi}+\varepsilon f_{\xi\xi\tau}.
\]
Consequently, the left-hand side of \eqref{eq2.29} becomes
\begin{align*}
	\Big(\partial_{tt}-\frac{\beta}{2}\partial_{xx}+\kappa\partial_t-\frac{\nu}{2}\partial_{xxt}\Big)f &=
	\Big(1-\frac{\beta}{2}\Big)f_{\xi\xi}+\kappa(-f_\xi+\varepsilon f_\tau)
	+\frac{\nu}{2}\Big(f_{\xi\xi\xi}-\varepsilon f_{\xi\xi\tau}\Big)
	-2\varepsilon f_{\xi\tau}
	+\mathcal O(\varepsilon^2)\\
	&=
	\Big(1-\frac{\beta}{2}\Big)f_{\xi\xi}-\kappa f_\xi+\frac{\nu}{2}f_{\xi\xi\xi}
	+\varepsilon\Big(\kappa f_\tau-2f_{\xi\tau}-\frac{\nu}{2}f_{\xi\xi\tau}\Big)
	+\mathcal O(\varepsilon^2).
\end{align*}
Next, we rewrite the right-hand side of \eqref{eq2.29} in $(\xi,\tau)$ variables. In far-field variables $\partial_x^{-1}$ becomes
$\partial_\xi^{-1}$ and, using $f_t=-f_\xi+\varepsilon f_\tau$ together with
\[
\partial_\xi^{-1}f_t=\partial_\xi^{-1}(-f_\xi+\varepsilon f_\tau)=-f+\varepsilon\,\partial_\xi^{-1}f_\tau,
\]
we obtain, that \eqref{eq:Ff-simplified} takes the form
\begin{align*}
	\mathcal F(f)
	&=\partial_\xi\bigg(
	-\kappa f^2
	+2 f f_\xi
	+\frac{\beta}{4}f f_\xi
	+\frac{\nu}{4}f_\xi^2
	-\frac{\nu}{4}f f_{\xi\xi}
	\bigg)
	+\mathcal O(\varepsilon).
\end{align*}
Therefore, inserting these expansions into \eqref{eq2.29} and neglecting $\mathcal O(\varepsilon^2)$ terms yields
\begin{align*}
	\Big(1-\frac{\beta}{2}\Big)f_{\xi\xi}-\kappa f_\xi+\frac{\nu}{2}f_{\xi\xi\xi}
	+\varepsilon\Big(\kappa f_\tau-2f_{\xi\tau}-\frac{\nu}{2}f_{\xi\xi\tau}\Big)
	=
	\varepsilon\,\partial_\xi\Big(
	-\kappa f^2+ \Big(2+\frac{\beta}{4}\Big) f f_\xi
	+\frac{\nu}{4}f_\xi^2
	-\frac{\nu}{4}f f_{\xi\xi}
	\Big).
\end{align*}
Finally, moving the $\varepsilon$-terms with $\tau$-derivatives to the left-hand side, we obtain the unidirectional
far-field equation
\begin{align*}
	\varepsilon\Big(\kappa-2\partial_\xi-\frac{\nu}{2}\partial_{\xi\xi}\Big)f_\tau
	&=
	-\Big(1-\frac{\beta}{2}\Big)\partial_{\xi\xi}f+\kappa\partial_\xi f-\frac{\nu}{2}\partial_{\xi\xi\xi}f
	+\varepsilon\,\partial_\xi\Big(
	-\kappa f^2+ \Big(2+\frac{\beta}{4}\Big) f f_\xi
	+\frac{\nu}{4}f_\xi^2
	-\frac{\nu}{4}f f_{\xi\xi}
	\Big),
\end{align*}
which is of $\mathcal O(\varepsilon^2)$ accuracy with respect to the bidirectional model \eqref{eq2.29}.
Using the trivial identities $\partial_\xi(f^2)=2ff_\xi$, $\partial_\xi(f_\xi^2)=2f_\xi f_{\xi\xi}$ and
$\partial_\xi(ff_{\xi\xi})=f_\xi f_{\xi\xi}+ff_{\xi\xi\xi}$, we can rewrite the previous equation in the form
\begin{align}\label{final:eq:asympt-structured}
	\varepsilon\Big(\kappa-2\partial_\xi-\frac{\nu}{2}\partial_{\xi\xi}\Big)f_\tau
	&=
	-\Big(1-\frac{\beta}{2}\Big)\partial_{\xi\xi}f+\kappa\partial_\xi f-\frac{\nu}{2}\partial_{\xi\xi\xi}f \nonumber\\
	&\quad
	+\varepsilon\Big(2+\frac{\beta}{4}\Big)(ff_\xi)_\xi
	+\varepsilon\,\frac{\nu}{2}\, f_\xi f_{\xi\xi}
	-\varepsilon\,\frac{\nu}{4}\,f f_{\xi\xi\xi}
	-2\varepsilon\kappa f f_\xi ,
\end{align}
where we have neglected $\mathcal O(\varepsilon^2)$ terms. Starting from \eqref{final:eq:asympt-structured}, we factor the operator on the left-hand side as
\[
\kappa-2\partial_\xi-\frac{\nu}{2}\partial_{\xi\xi}
=\left(\kappa-\frac{\nu}{2}\partial_{\xi\xi}\right)\left(\Id-\frac{2\partial_\xi}{\kappa-\frac{\nu}{2}\partial_{\xi\xi}}\right).
\]
Dividing by $\kappa-\frac{\nu}{2}\partial_{\xi\xi}$ and setting
\[
\mathcal{P}:=\left(\kappa-\frac{\nu}{2}\partial_{\xi\xi}\right)^{-1},
\]
where we recall that $\mathcal{P}$ is the inverse of the Helmholtz operator defined in \eqref{opP}, we obtain
\begin{align}\label{eq:preM}
	\varepsilon\left(\Id-\frac{2\partial_{\xi}}{\kappa-\frac{\nu}{2}\partial_{\xi\xi}}\right)f_{\tau}
	&=\mathcal{P}\bigg[
	- \Big(1-\frac{\beta}{2}\Big)\partial_{\xi\xi}f+\kappa\partial_{\xi}f-\frac{\nu}{2}\partial_{\xi\xi\xi}f \nonumber\\
	&\hspace{2.05cm}
	+\varepsilon\,\partial_\xi\Big(
	-\kappa f^2+\Big(2+\frac{\beta}{4}\Big) f f_\xi
	+\frac{\nu}{4}f_\xi^2-\frac{\nu}{4}f f_{\xi\xi}
	\Big)\bigg].
\end{align}

Applying the conjugate operator $\left(\Id+\frac{2\partial_{\xi}}{\kappa-\frac{\nu}{2}\partial_{\xi\xi}}\right)$ to both sides yields
\begin{align}\label{eq:conjugate}
	\varepsilon\left(\Id-4\left(\frac{\partial_{\xi}}{\kappa-\frac{\nu}{2}\partial_{\xi\xi}}\right)^{2}\right)f_{\tau}
	&=\left(\Id+\frac{2\partial_{\xi}}{\kappa-\frac{\nu}{2}\partial_{\xi\xi}}\right)\mathcal{P}\bigg[
	- \Big(1-\frac{\beta}{2}\Big)\partial_{\xi\xi}f+\kappa\partial_{\xi}f-\frac{\nu}{2}\partial_{\xi\xi\xi}f \nonumber\\
	&\hspace{2.05cm}
	+\varepsilon\,\partial_\xi\Big(
	-\kappa f^2+\Big(2+\frac{\beta}{4}\Big) f f_\xi
	+\frac{\nu}{4}f_\xi^2-\frac{\nu}{4}f f_{\xi\xi}
	\Big)\bigg].
\end{align}

Finally, inverting the operator $\Id-4(\partial_\xi\mathcal P)^2$ and using the definition of $\mathcal M$ in \eqref{opP},
namely that $\mathcal M=\big(\Id-4(\partial_\xi\mathcal P)^2\big)^{-1}\big(\Id+2\partial_\xi\mathcal P\big)$, we obtain
\begin{align}\label{eq:MP-final}
\varepsilon f_\tau
&=\mathcal M\,\mathcal P\bigg[
- \Big(1-\frac{\beta}{2}\Big)\partial_{\xi\xi}f+\kappa\partial_{\xi}f-\frac{\nu}{2}\partial_{\xi\xi\xi}f \nonumber\\
&\hspace{2.2cm}
+\varepsilon\,\partial_\xi\Big(
-\kappa f^2+\Big(2+\frac{\beta}{4}\Big) f f_\xi
+\frac{\nu}{4}f_\xi^2-\frac{\nu}{4}f f_{\xi\xi}
\Big)\bigg],
\end{align}
where we have neglected $\mathcal O(\varepsilon^2)$ terms. Thus, dividing both sides by $\varepsilon$ and reverting to $(x,t)$ variables for simplicity's sake we arrive to the asymptotic model \eqref{asymptotic:model:contr}.
\begin{remark}
In this article, the asymptotic models we derive are accurate up to $\mathcal{O}(\varepsilon^2)$, and hence all higher-order contributions are neglected. A natural direction for future work is to retain terms beyond quadratic order and investigate whether they produce genuinely new effects compared to the $\mathcal{O}(\varepsilon^2)$ approximation. 
We also note that, in the course of deriving the unidirectional asymptotic model \eqref{asymptotic:model:contr}, we obtained an alternative model with the same $\mathcal{O}(\varepsilon^2)$ precision, namely \eqref{eq2.29}--\eqref{eq:Ff-simplified}. While it would be interesting to study, for instance, the well-posedness of \eqref{eq2.29}--\eqref{eq:Ff-simplified}, we do not pursue this analysis here.
\end{remark}

\section{Well-posedness of the unidirectional asymptotic model}\label{sec:wp}
In this section, we show the local existence and uniqueness of strong solutions in Sobolev spaces of \eqref{asymptotic:model:contr}. More precisely, we show the following result:

\begin{theorem}\label{th:existence:strong}
	Let $\nu>0$ and $\kappa>0$, let $s>\frac{5}{2}$, and let $f_0\in H^s(\TT)$ be mean-zero initial data for \eqref{asymptotic:model:contr}.
	Then there exists $T>0$, depending only on $\|f_0\|_{H^s(\TT)}$ and the parameters of the model, and a unique strong solution
	\[
	f\in C([0,T];H^s(\TT)),\qquad f(\cdot,0)=f_0.
	\]
\end{theorem}
\begin{remark}
	The regularity threshold $s>\frac52$ in Theorem~\ref{th:existence:strong} is imposed to handle the highest--order nonlinearities
	in \eqref{asymptotic:model:contr}, in particular the cubic term $f\,f_{xxx}$ (and the related product $f_x f_{xx}$), which
	require additional control to define the nonlinearity and close the estimates in a strong sense.
	In the BBM regime $\nu=0$, these dispersive/quasilinear terms disappear and the equation reduces to a semilinear BBM--type model of lower differential order. As a consequence, one can prove local existence and uniqueness already for initial data
	$f_0\in H^s(\TT)$ with $s>\frac32$ (using the standard energy method and Sobolev algebra/embedding properties in one dimension).
	This better behavior is also reflected later in the paper: in Section~\ref{sec:bbm-global} we obtain a global existence and decay result 	for small $H^2$ data in the case $\nu=0$.
\end{remark}

\begin{proof}[Proof of Theorem \ref{th:existence:strong}]
The proof follows from the combination of standard priori energy estimates and the use of a suitable approximation procedure by means of mollifiers, cf. \cite{MajdaBertozzi2002}. Then, let us first focus on deriving the a priori estimates and later explain how to construct such solution via mollification. To conclude we will also show the uniqueness. In order to structure the proof properly we split it into several steps.

\subsubsection*{\underline{Step 1: conservation of the mean and a priori estimates}}

Let us first check that the mean is conserved along solutions, i.e.,
\begin{equation}\label{mean:prop}
	\int_{\TT} f(x,t)\,dx=\int_{\TT} f_{0}(x)\,dx.
\end{equation}

\begin{align}\label{asymptotic:model:contr2}
	f_{t}&=\mathcal{M}\mathcal{P}\bigg[ - \frac{1}{\varepsilon}\Big(1-\frac{\beta}{2}\Big)f_{xx}
	+ \frac{1}{\varepsilon}\kappa f_{x}
	- \frac{1}{\varepsilon}\frac{\nu}{2}f_{xxx}
	+ \Big(2+\frac{\beta}{4}\Big)(f f_{x})_{x}    \nonumber \\
	&\hspace{4.7cm}+ \frac{1}{\varepsilon}\frac{\nu}{4}f_{x}f_{xx}
	-\frac{\nu}{4}ff_{xxx}
	-2\kappa ff_{x}\bigg],
\end{align}

Indeed, rewriting \eqref{asymptotic:model:contr2} as
\begin{align*}
	f_{t}&=\mathcal{M}\mathcal{P}\partial_{x}\bigg[ - \frac{1}{\varepsilon}\Big(1-\frac{\beta}{2}\Big)f_{x}+ \frac{1}{\varepsilon}\kappa f
	- \frac{1}{\varepsilon}\frac{\nu}{2}f_{xx} + \Big(2+\frac{\beta}{4}\Big)f f_{x}
	+ \frac{1}{\varepsilon}\frac{\nu}{8}f_{x}^{2} -\kappa f^{2}\bigg]
	-\frac{\nu}{4}\mathcal{M}\mathcal{P}\Big(ff_{xxx}\Big),
\end{align*}
we obtain (with our Fourier convention)
\[
\widehat{f_t}(0,t)
=-\frac{\nu}{4}\,\widehat{\mathcal{M}\mathcal{P}}(0)\,\widehat{ff_{xxx}}(0,t)
= -\frac{\nu}{4}\,\widehat{\mathcal{M}\mathcal{P}}(0)\int_{\TT} f f_{xxx}\,dx = 0,
\]
since $\int_{\TT} f f_{xxx}\,dx=0$ by periodicity. Moreover,
$\widehat{\mathcal{M}\mathcal{P}}(0)=1/\kappa$ by \eqref{multiplierP-multiplierM}.
Therefore $\widehat{f}(0,t)=\widehat{f}(0,0)$, which is equivalent to \eqref{mean:prop}. \medskip

We define an energy $\mathcal{E}(t)=\norm{f}_{L^{2}(\TT)}^{2}+ \norm{\Lambda^{s}f}_{L^{2}(\TT)}^{2}$ which is of course equivalent to the square of the $H^{s}$ norm of $f$. In the sequel 
we will show that for $s>\frac{5}{2}$ the energy satisfies
\[
\frac{d}{dt}\mathcal{E}(t)\le C\big(\mathcal{E}(t)+\mathcal{E}(t)^{3/2}\big).
\]

 We begin by estimating the evolution of the $L^{2}$ norm of $f$. Testing equation \eqref{asymptotic:model:contr} with $f$ and integrating by parts we find that
\[
\frac{1}{2}\frac{d}{dt}\norm{f}_{L^{2}}^{2}= I_{1}+I_{2} 
\]
with 
\[I_{1}= \int_{\TT} \mathcal{M}\mathcal{P}\bigg[ - \frac{1}{\varepsilon}(1-\frac{\beta}{2})f_{xx}+ \frac{1}{\varepsilon}\kappa f_{x}- \frac{1}{\varepsilon}\frac{\nu}{2}f_{xxx}\bigg] f \ dx,\]
\[I_{2}=\int_{\TT} \mathcal{M}\mathcal{P}\bigg[\left(2+\frac{\beta}{4}\right)(f f_{x})_{x} + \frac{1}{\varepsilon}\frac{\nu}{4}f_{x}f_{xx}  -\frac{\nu}{4}ff_{xxx}-2\kappa ff_{x}\bigg]f \ dx. \]
Then, by means of Lemma \eqref{lemma:op} we may immediately bound $I_{1}$ as
\begin{align*}
\abs{I_{1}}&\lesssim \left(\norm{\mathcal{M}\mathcal{P}f_{xx}}_{L^{2}}+\norm{\mathcal{M}\mathcal{P}f_{x}}_{L^{2}}+\norm{\mathcal{M}\mathcal{P}f_{xxx}}_{L^{2}}\right)\norm{f}_{L^{2}} \\
&\lesssim  \norm{f}_{L^{2}}^{2}+\norm{f_{x}}_{L^{2}}\norm{f}_{L^2}\lesssim  \norm{f}_{L^{2}}^{2}+\norm{f_{x}}_{L^{2}}^{2}\lesssim \mathcal{E}(t).
\end{align*}
In order to bound the nonlinear terms in $I_{2}$ we use the decomposition $\mathcal{M}=(\mathsf{Id}+\mathcal{S})$ to write
\begin{align*}
I_{2}=\int_{\TT}\mathcal{P}\bigg[\left(2+\frac{\beta}{4}\right)(f f_{x})_{x} + \frac{1}{\varepsilon}\frac{\nu}{4}f_{x}f_{xx} -\frac{\nu}{4}ff_{xxx}-2\kappa ff_{x}\bigg]f \ dx + \mbox{l.o.t}
\end{align*}
The lower order terms are even easier to bound since $\mathcal{S}$ is a smoothing operator of order $-1$.  Integrating by parts several times and using the fact that $\mathcal{P}$ is self-adjoint together with bound \eqref{estimate:P:smooth} of Lemma \ref{lemma:op} we find that
\begin{align*}
\abs{I_{2}}&\lesssim \norm{f}_{L^{2}}^{2}\norm{f_{x}}_{L^{2}}+\norm{f_{x}}_{L^{2}}^{2}\norm{f}_{L^{2}}+\norm{f}_{L^{2}}^{3}\lesssim  \mathcal{E}^{3/2}(t).
\end{align*}
Thus,
\begin{equation}\label{estimate:L2}
\frac{d}{dt}\norm{f}_{L^{2}}^{2}\lesssim \left( \mathcal{E}(t)+\mathcal{E}^{3/2}(t) \right).
\end{equation}
To derive the evolution of the higher order norm, we apply $\Lambda^s$ to \eqref{asymptotic:model:contr2} and take the $L^2$ inner product with $\Lambda^s f$ to find that

\[
\frac{1}{2}\frac{d}{dt}\norm{\Lambda^{s}f}_{L^{2}}^{2}= J_{1}+J_{2}
\]
with
\[J_{1}= \int_{\TT}  \mathcal{M}\mathcal{P}\bigg[ - \frac{1}{\varepsilon}(1-\frac{\beta}{2})f_{xx}+ \frac{1}{\varepsilon}\kappa f_{x}- \frac{1}{\varepsilon}\frac{\nu}{2}f_{xxx}\bigg] \Lambda^{2s} f \ dx,\]
\[J_{2}=\int_{\TT} \mathcal{M}\mathcal{P}\bigg[\left(2+\frac{\beta}{4}\right)(f f_{x})_{x} + \frac{1}{\varepsilon}\frac{\nu}{4}f_{x}f_{xx}  -\frac{\nu}{4}ff_{xxx}-2\kappa ff_{x}\bigg] \Lambda^{2s} f \ dx. \]
Let us first deal with the linear terms in $J_{1}$. Using once again that $\mathcal{M}=(\mathsf{Id}+\mathcal{S})$ we write
\begin{align*}
J_{1}&=\int_{\TT} \mathcal{P}\bigg[ - \frac{1}{\varepsilon}(1-\frac{\beta}{2})f_{xx}+ \frac{1}{\varepsilon}\kappa f_{x}- \frac{1}{\varepsilon}\frac{\nu}{2}f_{xxx}\bigg] \Lambda^{2s} f \ dx+ \mbox{l.o.t}\\
&= J_{11}+J_{12}+J_{13}+ \mbox{l.o.t}
\end{align*}
Using \eqref{twoder:helm} of Lemma \eqref{lemma:op} we find that 
\[ J_{11}= - \frac{2}{\nu\varepsilon}(1-\frac{\beta}{2}) \int_{\TT} f \Lambda^{2s} f \ dx +  \frac{2\kappa }{\nu\varepsilon}(1-\frac{\beta}{2})\int_{\TT} \mathcal{P}f \Lambda^{2s} f \ dx \lesssim \norm{\Lambda^{s}f}_{L^{2}}^{2}. \]
On the other hand, since $\Lambda^{s}$ and $\mathcal{P}$ are self-adjoint we obtain that 
\[ J_{12}= \frac{\kappa}{2\varepsilon} \int_{\TT} \partial_{x}\left(\Lambda^{s}\mathcal{P}^{1/2} f \right)^{2} \ dx = 0, \quad  J_{13}= \frac{\nu}{4\varepsilon} \int_{\TT} \partial_{x}\left(\Lambda^{s}\mathcal{P}^{1/2} f_{x}\right)^{2} \ dx=0.\]
Since there has been a cancellation in $J_{12}, J_{13}$ due to the special structure, it is important to check that the \textit{lower order terms} coming from the operator $\mathcal{S}$ (where such cancellation is not available) can be bounded. Indeed, both terms are given by
\[ J_{12}^{low}=\frac{1}{\varepsilon}\kappa \int_{\TT} \mathcal{S}\mathcal{P}  f_{x} \Lambda^{2s} f \ dx, \quad J_{13}^{low}= - \frac{1}{\varepsilon}\frac{\nu}{2}\int_{\TT} \mathcal{S}\mathcal{P}  f_{xxx} \Lambda^{2s} f \ dx.\]
Therefore, by means of \eqref{estimate:P:smooth} and \eqref{bound:smooth} in Lemma  \ref{lemma:op} 
\[ \abs{J_{12}^{low}}\lesssim  \norm{\mathcal{S}\mathcal{P} \Lambda^{s} f_{x}}_{L^{2}}\norm{\Lambda^{s}f}_{L^{2}}\lesssim \norm{\Lambda^{s}f}_{L^{2}}^{2}, \quad  \abs{J_{13}^{low}}\lesssim  \norm{\mathcal{S}\mathcal{P} \Lambda^{s} f_{xxx}}_{L^{2}}\norm{\Lambda^{s}f}_{L^{2}}\lesssim \norm{\Lambda^{s}f}_{L^{2}}^{2}.\]
Hence, combining the previous estimates we have shown that
\begin{equation}\label{estimate:J1}
\abs{J_{1}}=\abs{J_{11}+J_{12}+J_{13}+ \mbox{l.o.t}} \lesssim \norm{\Lambda^{s}f}_{L^{2}}^{2}.
\end{equation}
The bound the non-linear terms in $J_{2}$ we write again $\mathcal{M}=(\mathsf{Id}+\mathcal{S})$ so that
\begin{align*}
J_{2}&=\int_{\TT}\mathcal{P}\bigg[\left(2+\frac{\beta}{4}\right)(f f_{x})_{x} + \frac{1}{\varepsilon}\frac{\nu}{4}f_{x}f_{xx}  -\frac{\nu}{4}ff_{xxx}-2\kappa ff_{x}\bigg] \Lambda^{2s} f \ dx+ \mbox{l.o.t} \\
&= J_{21}+J_{22}+J_{23}+J_{24}+ \mbox{l.o.t}
\end{align*}
Integrating by parts and using identity \eqref{twoder:helm} we have that
\begin{align*}
 J_{21}=\frac{\left(2+\frac{\beta}{4}\right)}{2}  \int_{\TT } f^{2} \Lambda^{2s} \mathcal{P} f_{xx} \ dx&= \frac{\left(2+\frac{\beta}{4}\right)}{\nu} \int_{\TT }   \Lambda^{s}(f^{2}) \Lambda^{s}f \ dx - \frac{\kappa\left(2+\frac{\beta}{4}\right)}{\nu}  \int_{\TT }\Lambda^{s}(f^{2}) \Lambda^{s}\mathcal{P}f \ dx \\
 & \quad \quad \lesssim \norm{\Lambda^{s}(f^{2})}_{L^{2}}\norm{\Lambda^{s}f}_{L^{2}}\lesssim \norm{\Lambda^{s}f}_{L^{2}}^{3}.
\end{align*}
Similarly,
\begin{align*}
 J_{22}=-\frac{1}{\varepsilon}\frac{\nu}{4} \int_{\TT } f_{x}^{2} \Lambda^{2s} \mathcal{P} f_{x} \ dx&= -\frac{1}{\varepsilon}\frac{\nu}{4}\int_{\TT }   \mathcal{P}^{1/2}\Lambda^{s}(f_{x}^{2}) \mathcal{P}^{1/2} \Lambda^{s}f_{x} \ dx \\
  & \lesssim \norm{\mathcal{P}^{1/2}\Lambda^{s}(f_{x}^{2})}_{L^{2}}\norm{\mathcal{P}^{1/2}\Lambda^{s}f_{x}}_{L^{2}}\lesssim \norm{\Lambda^{s}f}_{L^{2}}^{3}.
\end{align*}
and
\begin{align*}
 J_{24}=\kappa \int_{\TT } f^{2} \Lambda^{2s} \mathcal{P} f_{x} \ dx&=\kappa \int_{\TT } \Lambda^{s} (f^{2}) \Lambda^{s} \mathcal{P} f_{x} \ dx \\
  & \lesssim \norm{\Lambda^{s}(f^{2})}_{L^{2}}\norm{\mathcal{P}\Lambda^{s}f_{x}}_{L^{2}}\lesssim \norm{\Lambda^{s}f}_{L^{2}}^{3}.
\end{align*}
The most singular term is $J_{23}$. Integrating twice by parts, we find that
\begin{align*}
J_{23}=-\frac{3\nu}{8}\int_{\TT} f_{x}f_{x} \Lambda^{2s}\mathcal{P}f_{x}\ dx- \frac{\nu}{4}\int_{\TT} ff_{x} \Lambda^{2s}\mathcal{P}f_{xx} \ dx
\end{align*}
The first term is identical to $J_{22}$ and hence
\[ \abs{\frac{3\nu}{8}\int_{\TT} f_{x}f_{x} \Lambda^{2s}\mathcal{P}f_{x}\ dx}\lesssim  \norm{\Lambda^{s}f}_{L^{2}}^{3}.\]
To bound the second term, we use identity \eqref{twoder:helm}  to write
\[  -\frac{\nu}{4}\int_{\TT} ff_{x} \Lambda^{2s}\mathcal{P}f_{xx} \ dx= -\frac{1}{2}\int_{\TT} ff_{x} \Lambda^{2s}f \ dx+ \frac{\kappa}{2}\int_{\TT} ff_{x} \Lambda^{2s}\mathcal{P}f \ dx,\]
where the latter integral is after integration by parts identical to $J_{24}$. To estimate the former, we use the Kato-Ponce commutator estimate \eqref{KPcom}  and Sobolev embedding \eqref{Sob:e}, namely
\begin{align*}
-\frac{1}{2}\int_{\TT} ff_{x} \Lambda^{2s}f \ dx&=\frac{1}{4}\int_{\TT} f_{x} |\Lambda^{s}f |^{2} \ dx-\frac{1}{2}\int_{\TT} [\Lambda^{s},f]f_{x}  \Lambda^{s}f \ dx \\
&\lesssim \norm{f_{x}}_{L^{\infty}} \norm{\Lambda^{s}f}^{2}_{L^2}+\norm{ [\Lambda^{s},f]f_{x}}_{L^2}\norm{\Lambda^{s}f}_{L^2} \\
&\lesssim \norm{f_{x}}_{L^{\infty}}\norm{\Lambda^{s}f}^{2}_{L^2} \lesssim \norm{\Lambda^{s}f}^{3}_{L^2}.
\end{align*}
and hence
\[ \abs{J_{23}}\lesssim \norm{\Lambda^{s}f}^{3}_{L^2}.\]
Combining the previous bounds we conclude that
\begin{equation}\label{estimate:J2}
\abs{J_{2}} =\abs{ J_{21}+J_{22}+J_{23}+J_{24}+ \mbox{l.o.t}} \lesssim  \norm{\Lambda^{s}f}^{3}_{L^2}.
\end{equation}
Thus, \eqref{estimate:J1}-\eqref{estimate:J2} yields
\begin{equation}\label{estimate:Hs}
 \frac{d}{dt}\norm{\Lambda^{s}f}_{L^{2}}^{2}\lesssim \left(\norm{\Lambda^{s}f}^{2}_{L^2}+\norm{\Lambda^{s}f}^{3}_{L^2}
\right)\lesssim \left( \mathcal{E}(t)+ \mathcal{E}^{3/2}(t)\right)
 \end{equation}
Both bounds \eqref{estimate:L2} and \eqref{estimate:Hs} lead to the following differential inequality
\begin{equation}\label{energy:ine}
\frac{d}{dt}\mathcal{E}(t)\le C\big(\mathcal{E}(t)+\mathcal{E}(t)^{3/2}\big).
\end{equation}
where $C=C(\kappa,\varepsilon,\beta,\nu)$ is a positive constant.  With the previous differential inequality at hand, one can show a uniform time of existence. Indeed, let $\overline{c}>0$ be such that $\mathcal{E}(0)\leq \overline{c}$. We want calculate for which values of $t>0$ we may guarantee that $\mathcal{E}(t)\leq 2\overline{c}$ and hence for such values we have that
\[ \mathcal{E}(t)\leq C \left(2\overline{c}+(2\overline{c})^{3/2} \right)t + \overline{c}.\]
Therefore, we conclude that $\mathcal{E}(t)\leq 2\overline{c}$ for all $t$ satisfying
\[ t\in \bigg[0, \frac{\overline{c}}{C \left(2\overline{c}+(2\overline{c})^{3/2} \right)}\bigg]. \]
\subsubsection*{\underline{Step 2: the approximate system, uniform estimates, and passage to the limit}}

Fix $\delta>0$ and assume $s>\frac52+\delta$. Let $\mathcal{J}^\epsilon$ be the periodic heat-kernel mollifier, i.e.
\[
\widehat{\mathcal{J}^\epsilon f}(k)=e^{-\epsilon |k|^2}\,\widehat{f}(k),\qquad k\in\mathbb{Z}.
\]
Then $\mathcal{J}^\epsilon$ is a self-adjoint Fourier multiplier, bounded on $H^r(\TT)$ for every $r\in\mathbb{R}$, it commutes
with $\Lambda^s$, $\mathcal{P}$ and $\mathcal{M}$, and it satisfies
\begin{equation}\label{eq:J-properties}
	\|\mathcal{J}^\epsilon f\|_{H^r}\le \|f\|_{H^r},
	\qquad
	\mathcal{J}^\epsilon f\to f \ \text{in }H^r \ \text{as }\epsilon\to 0^+ \ \text{for all }r\in\mathbb{R}.
\end{equation}

We consider the mollified Cauchy problem
\begin{align}\label{eq:mollified-system-step2}
	f^\epsilon_t
	&=\mathcal{J}^{\epsilon}\mathcal{M}\mathcal{P}\bigg[
	- \frac{1}{\varepsilon}\Big(1-\frac{\beta}{2}\Big)\mathcal{J}^{\epsilon} f^{\epsilon}_{xx}
	+ \frac{\kappa}{\varepsilon}\mathcal{J}^{\epsilon} f^{\epsilon}_{x}
	- \frac{\nu}{2\varepsilon}\mathcal{J}^{\epsilon}f^{\epsilon}_{xxx}
	+ \Big(2+\frac{\beta}{4}\Big)\big(\mathcal{J}^{\epsilon}f^{\epsilon}\,\mathcal{J}^{\epsilon}f^{\epsilon}_{x}\big)_{x}
	\nonumber\\
	&\hspace{2.6cm}
	+ \frac{\nu}{4\varepsilon}\,\mathcal{J}^{\epsilon}f^{\epsilon}_{x}\,\mathcal{J}^{\epsilon}f^{\epsilon}_{xx}
	- \frac{\nu}{4}\,\mathcal{J}^{\epsilon}f^{\epsilon}\,\mathcal{J}^{\epsilon}f^{\epsilon}_{xxx}
	-2\kappa\,\mathcal{J}^{\epsilon} f^{\epsilon}\,\mathcal{J}^{\epsilon} f^{\epsilon}_{x}
	\bigg],
\end{align}
with $f^\epsilon(\cdot,0)=f_0\in H^s(\TT)$.
For each fixed $\epsilon>0$, the right-hand side is a locally Lipschitz map $H^s(\TT)\to H^s(\TT)$; hence, by Picard's theorem,
there exists $T_\epsilon>0$ and a unique solution $f^\epsilon\in C^1([0,T_\epsilon];H^s(\TT))$. \medskip

Repeating verbatim the energy estimates of Step~1 (using that $\mathcal{J}^\epsilon$ is self-adjoint and commutes with
$\Lambda^s$, $\mathcal{P}$ and $\mathcal{M}$), we obtain the differential inequality
\begin{equation}\label{eq:energy-eps}
	\frac{d}{dt}\widetilde{\mathcal{E}}^\epsilon(t)
	\le C\Big(\widetilde{\mathcal{E}}^\epsilon(t)+\big(\widetilde{\mathcal{E}}^\epsilon(t)\big)^{3/2}\Big),
	\qquad
	\widetilde{\mathcal{E}}^\epsilon(t):=\|f^\epsilon(t)\|_{L^2}^2+\|\Lambda^s f^\epsilon(t)\|_{L^2}^2,
\end{equation}
with $C=C(\varepsilon,\kappa,\beta,\nu)$ independent of $\epsilon$. Consequently, there exists $T>0$, depending only on
$\|f_0\|_{H^s}$ and the parameters, such that each $f^\epsilon$ exists on $[0,T]$ and satisfies the uniform bound
\begin{equation}\label{eq:uniform-Hs-bound}
	\sup_{t\in[0,T]}\|f^\epsilon(t)\|_{H^s}\le C_0,
\end{equation}
for some $C_0$ independent of $\epsilon$.

We next control the time derivative uniformly. Let $\mathcal{F}^\epsilon(f^\epsilon)$ denote the bracket in
\eqref{eq:mollified-system-step2}, so that
\begin{equation}\label{eq:ft-form}
	f^\epsilon_t=\mathcal{J}^\epsilon\mathcal{M}\mathcal{P}\big(\mathcal{F}^\epsilon(f^\epsilon)\big).
\end{equation}
We claim that $\mathcal{F}^\epsilon(f^\epsilon)$ is uniformly bounded in $L^\infty(0,T;H^{s-3})$. Indeed, by
\eqref{eq:J-properties} and \eqref{eq:uniform-Hs-bound},
\[
\|\mathcal{J}^\epsilon f^\epsilon_{xx}\|_{H^{s-3}}+\|\mathcal{J}^\epsilon f^\epsilon_x\|_{H^{s-3}}
+\|\mathcal{J}^\epsilon f^\epsilon_{xxx}\|_{H^{s-3}}
\lesssim \|f^\epsilon\|_{H^s}\le C_0.
\]
Moreover, since $s>\frac52+\delta$, we have $H^{s-1}(\TT)\hookrightarrow W^{1,\infty}(\TT)$, hence
\[
\|\mathcal{J}^\epsilon f^\epsilon\|_{L^\infty}+\|\mathcal{J}^\epsilon f^\epsilon_x\|_{L^\infty}
+\|\mathcal{J}^\epsilon f^\epsilon_{xx}\|_{L^\infty}
\lesssim \|f^\epsilon\|_{H^s}\le C_0.
\]
Since multiplication by an $L^\infty$ function is bounded on $H^{r}(\TT)$ for any $r\in\mathbb{R}$, we obtain for example
\[
\|\mathcal{J}^\epsilon f^\epsilon\,\mathcal{J}^\epsilon f^\epsilon_{xxx}\|_{H^{s-3}}
\lesssim \|\mathcal{J}^\epsilon f^\epsilon\|_{L^\infty}\,\|\mathcal{J}^\epsilon f^\epsilon_{xxx}\|_{H^{s-3}}
\lesssim \|f^\epsilon\|_{H^s}^2\le C_0^2,
\]
and similarly for the other nonlinearities. Hence
\begin{equation}\label{eq:F-uniform}
	\sup_{t\in[0,T]}\|\mathcal{F}^\epsilon(f^\epsilon(t))\|_{H^{s-3}}\le C,
\end{equation}
with $C$ independent of $\epsilon$. Since $\mathcal{P}$ is of order $-2$ and $\mathcal{M}$ and $\mathcal{J}^\epsilon$ are of order $0$,
\eqref{eq:ft-form} and \eqref{eq:F-uniform} yield
\begin{equation}\label{eq:dt-uniform}
	\sup_{t\in[0,T]}\|f^\epsilon_t(t)\|_{H^{s-1}}\le C,
\end{equation}
with $C$ independent of $\epsilon$. In particular, $\{f^\epsilon\}_{\epsilon>0}$ is equicontinuous in time with values in $H^{s-1}$.

Choose $s'$ such that $\max\{\frac52,\,s-1\}<s'<s$. By the compact embedding $H^s(\TT)\hookrightarrow\hookrightarrow H^{s'}(\TT)$,
the continuous embedding $H^{s'}(\TT)\hookrightarrow H^{s-1}(\TT)$, and the uniform bounds
\eqref{eq:uniform-Hs-bound}--\eqref{eq:dt-uniform}, the Arzel\`a--Ascoli theorem (equivalently, the Aubin--Lions lemma) yields a subsequence
(not relabeled) and a limit $f$ such that
\begin{equation}\label{eq:strong-conv}
	f^\epsilon\to f \quad\text{in } C([0,T];H^{s'}(\TT)).
\end{equation}
Moreover, by Banach--Alaoglu,
\begin{equation}\label{eq:weakstar-conv}
	f^\epsilon \rightharpoonup^\ast f \quad\text{in } L^\infty(0,T;H^s(\TT)).
\end{equation}
In particular, $f\in L^\infty(0,T;H^s(\TT))$ and $f(\cdot,t)\in H^s(\TT)$ for all $t\in[0,T]$.

We now pass to the limit in \eqref{eq:mollified-system-step2}. Let $\mathcal{F}(f)$ denote the non-mollified bracket in
\eqref{asymptotic:model:contr}, namely
\[
\mathcal{F}(f):=
- \frac{1}{\varepsilon}\Big(1-\frac{\beta}{2}\Big) f_{xx}
+ \frac{\kappa}{\varepsilon} f_x
- \frac{\nu}{2\varepsilon} f_{xxx}
+ \Big(2+\frac{\beta}{4}\Big)(f f_x)_x
+ \frac{\nu}{4\varepsilon} f_x f_{xx}
- \frac{\nu}{4} f f_{xxx}
- 2\kappa f f_x .
\]
Fix any $s'\in\big(\frac52,\,s\big)$. From \eqref{eq:strong-conv} and \eqref{eq:J-properties} we have
\[
\mathcal{J}^\epsilon f^\epsilon \to f
\quad \text{in } C([0,T];H^{s'}(\TT)).
\]
Since $\mathcal{J}^\epsilon$ commutes with $\partial_x$, by continuity of $\partial_x:H^r\to H^{r-1}$ we also obtain
\[
\mathcal{J}^\epsilon f^\epsilon_x \to f_x \quad \text{in } C([0,T];H^{s'-1}(\TT)),\qquad
\mathcal{J}^\epsilon f^\epsilon_{xx} \to f_{xx} \quad \text{in } C([0,T];H^{s'-2}(\TT)),
\]
and
\[
\mathcal{J}^\epsilon f^\epsilon_{xxx} \to f_{xxx} \quad \text{in } C([0,T];H^{s'-3}(\TT)).
\]
Moreover, since $s'>\frac52$, we have $H^{s'-1}(\TT)\hookrightarrow L^\infty(\TT)$, hence
\[
\sup_{t\in[0,T]}\|\mathcal{J}^\epsilon f^\epsilon(t)\|_{L^\infty}
+\sup_{t\in[0,T]}\|f(t)\|_{L^\infty}
\;<\;\infty .
\]
We can therefore pass to the limit in each nonlinear product in the space $H^{s'-3}(\TT)$.
Indeed, we use the standard fact that multiplication by an $L^\infty$ function acts continuously on $H^r(\TT)$ for any
$r\in\mathbb{R}$, i.e.
\[
\|uv\|_{H^r}\lesssim \|u\|_{L^\infty}\|v\|_{H^r}\qquad (u\in L^\infty,\ v\in H^r).
\]
For instance,
\[
\|\mathcal{J}^\epsilon f^\epsilon\,\mathcal{J}^\epsilon f^\epsilon_{xxx}- f\,f_{xxx}\|_{H^{s'-3}}
\le
\|(\mathcal{J}^\epsilon f^\epsilon-f)\,\mathcal{J}^\epsilon f^\epsilon_{xxx}\|_{H^{s'-3}}
+\|f\,(\mathcal{J}^\epsilon f^\epsilon_{xxx}-f_{xxx})\|_{H^{s'-3}},
\]
and each term tends to zero in $C([0,T];H^{s'-3})$ because
$\mathcal{J}^\epsilon f^\epsilon-f\to0$ in $C([0,T];H^{s'})\hookrightarrow C([0,T];L^\infty)$ and
$\mathcal{J}^\epsilon f^\epsilon_{xxx}\to f_{xxx}$ in $C([0,T];H^{s'-3})$.
Similarly,
\[
\big(\mathcal{J}^\epsilon f^\epsilon\,\mathcal{J}^\epsilon f^\epsilon_x\big)_x \to (f f_x)_x
\quad\text{in } C([0,T];H^{s'-3}(\TT)),
\]
and the remaining nonlinearities are treated in the same way. Consequently,
\begin{equation}\label{eq:F-conv}
	\mathcal{F}^\epsilon(f^\epsilon)\to \mathcal{F}(f)
	\quad\text{in } C([0,T];H^{s'-3}(\TT)).
\end{equation}
Applying the Fourier multipliers $\mathcal{P}$ and $\mathcal{M}$, and using that $\mathcal{P}$ is of order $-2$, we infer
\[
\mathcal{M}\mathcal{P}\big(\mathcal{F}^\epsilon(f^\epsilon)\big)\to \mathcal{M}\mathcal{P}\big(\mathcal{F}(f)\big)
\quad\text{in } C([0,T];H^{s'-1}(\TT)).
\]
Since $\mathcal{J}^\epsilon\to \Id$ strongly on $H^{s'-1}(\TT)$, it follows that
\[
\mathcal{J}^\epsilon\mathcal{M}\mathcal{P}\big(\mathcal{F}^\epsilon(f^\epsilon)\big)\to
\mathcal{M}\mathcal{P}\big(\mathcal{F}(f)\big)
\quad\text{in } C([0,T];H^{s'-1}(\TT)).
\]
Writing \eqref{eq:mollified-system-step2} in integral form,
\[
f^\epsilon(t)=f_0+\int_0^t \mathcal{J}^\epsilon\mathcal{M}\mathcal{P}\big(\mathcal{F}^\epsilon(f^\epsilon(r))\big)\,dr,
\]
we may pass to the limit as $\epsilon\to 0^+$ to obtain
\[
f(t)=f_0+\int_0^t \mathcal{M}\mathcal{P}\big(\mathcal{F}(f(r))\big)\,dr
\quad\text{in } H^{s'-1}(\TT).
\]
In particular, $f\in C^1([0,T];H^{s'-1}(\TT))$ and $f$ satisfies \eqref{asymptotic:model:contr} in $H^{s'-1}(\TT)$
for all $t\in[0,T]$.
Finally, since $f\in L^\infty(0,T;H^s)$ and $f_t=\mathcal{M}\mathcal{P}\big(\mathcal{F}(f)\big)\in L^\infty(0,T;H^{s-1})$,
we have $f\in W^{1,\infty}(0,T;H^{s-1})\subset C([0,T];H^{s-1})$. The upgrade from
$f\in L^\infty(0,T;H^s)\cap C([0,T];H^{s-1})$ to $f\in C([0,T];H^s)$ follows by a standard argument:
one first obtains weak continuity $f\in C_w([0,T];H^s)$ and then proves continuity of $t\mapsto \|f(t)\|_{H^s}$, see \cite{MajdaBertozzi2002}.

\subsubsection*{\underline{Step 3: Uniqueness of solutions}}

Let $f,g\in C([0,T];H^s(\TT))$ with $s>\frac52$ be two strong solutions of \eqref{asymptotic:model:contr}
with the same initial datum $f_0$. Set $w=f-g$. Since $\mathcal{M}$ and $\mathcal{P}$ commute, $w$ satisfies
\begin{equation}\label{eq:w-eq}
	w_t = \mathcal{M}\mathcal{P}\big(\mathcal{F}(f)-\mathcal{F}(g)\big),
\end{equation}
where
\[
\mathcal{F}(f):=
- \frac{1}{\varepsilon}\Big(1-\frac{\beta}{2}\Big) f_{xx}
+ \frac{\kappa}{\varepsilon} f_x
- \frac{\nu}{2\varepsilon} f_{xxx}
+ \Big(2+\frac{\beta}{4}\Big)(f f_x)_x
+ \frac{\nu}{4\varepsilon} f_x f_{xx}
- \frac{\nu}{4} f f_{xxx}
- 2\kappa f f_x.
\]
Since $\mathcal{M}$ and $\mathcal{P}$ are Fourier multipliers, we may equivalently apply $\mathcal{P}^{-1}$ to both sides,
obtaining
\begin{equation}\label{eq:w-Pinv}
	\mathcal{P}^{-1} w_t = \mathcal{M}\big(\mathcal{F}(f)-\mathcal{F}(g)\big),
	\qquad \mathcal{P}^{-1}=\kappa-\frac{\nu}{2}\partial_{xx}.
\end{equation}
Testing \eqref{eq:w-Pinv} with $w$ in $L^2(\TT)$ and using the self-adjointness of $\mathcal{P}^{-1}$ yields
\[
\frac12\frac{d}{dt}\langle \mathcal{P}^{-1}w,w\rangle
= \langle \mathcal{M}(\mathcal{F}(f)-\mathcal{F}(g)),\,w\rangle.
\]
Defining the energy
\[
E_{w}(t):=\kappa\|w\|_{L^2}^2+\frac{\nu}{2}\|w_x\|_{L^2}^2,
\]
we have $E_w(t)\simeq \|w(t)\|_{H^1}^2$. Since $\mathcal{M}=\Id+\mathcal{S}$ with $\mathcal{S}$ a smoothing operator of order $-1$,
we may absorb its contribution into constants and focus on the $\Id$ part:
\[
\frac{d}{dt}E_w(t)
\;\lesssim\;
\big|\langle \mathcal{F}(f)-\mathcal{F}(g),\,w\rangle\big|.
\]
For the linear part of $\mathcal{F}$, one has
\[
\mathcal{F}_{\mathrm{lin}}(f)-\mathcal{F}_{\mathrm{lin}}(g)
= -\frac{1}{\varepsilon}\Big(1-\frac{\beta}{2}\Big)w_{xx}
+\frac{\kappa}{\varepsilon}w_x
-\frac{\nu}{2\varepsilon}w_{xxx}.
\]
Hence, integrating by parts and using periodicity,
\[
\langle \mathcal{F}_{\mathrm{lin}}(f)-\mathcal{F}_{\mathrm{lin}}(g),w\rangle
= \frac{1}{\varepsilon}\Big(1-\frac{\beta}{2}\Big)\|w_x\|_{L^2}^2
\]
so that
\[
\big|\langle \mathcal{F}_{\mathrm{lin}}(f)-\mathcal{F}_{\mathrm{lin}}(g),w\rangle\big|
\;\lesssim\;
\|w\|_{H^1}^2,
\]

We now write the nonlinear part of the bracket as
\[
\mathcal{F}_{\mathrm{nl}}(h)
:=\Big(2+\frac{\beta}{4}\Big)(h h_x)_x
+\frac{\nu}{4\varepsilon}h_x h_{xx}
+\frac{\nu}{4}h\,h_{xxx}
-2\kappa h h_x ,
\]
so that
\[
\mathcal{F}_{\mathrm{nl}}(f)-\mathcal{F}_{\mathrm{nl}}(g)
=\Big(2+\frac{\beta}{4}\Big)M_1+\frac{\nu}{4\varepsilon}M_2+\frac{\nu}{4}M_3-2\kappa M_4,
\]
where, with $w:=f-g$,
\[
M_1:=(f f_x)_x-(g g_x)_x,\qquad
M_2:=f_x f_{xx}-g_x g_{xx},\qquad
M_3:=f f_{xxx}-g g_{xxx},\qquad
M_4:=f f_x-g g_x.
\]
Set $K(t):=\|f(t)\|_{H^s}+\|g(t)\|_{H^s}$; since $s>\frac52$, we have the embeddings
$H^s(\TT)\hookrightarrow W^{2,\infty}(\TT)$ and therefore $\|f\|_{W^{2,\infty}}+\|g\|_{W^{2,\infty}}\lesssim K(t)$.

\medskip
We start with $M_1$. Testing against $w$ and integrating by parts gives
\[
\langle M_1,w\rangle
=\int_\TT (f_x+g_x)w_x w\,dx+\int_\TT (f+g)w_{xx}w\,dx
=\frac12\int_\TT (f_x+g_x)w^2\,dx-\int_\TT (f+g)w_x^2\,dx,
\]
and hence, by Hölder,
\[
|\langle M_1,w\rangle|
\lesssim \big(\|f_x\|_{L^\infty}+\|g_x\|_{L^\infty}+\|f\|_{L^\infty}+\|g\|_{L^\infty}\big)\,\|w\|_{H^1}^2
\lesssim K(t)\,\|w\|_{H^1}^2.
\]

\medskip
For $M_2$ we write
\[
M_2=f_x f_{xx}-g_x g_{xx}=f_x w_{xx}+g_{xx}w_x.
\]
Then, integrating by parts once in the first term,
\[
\langle f_x w_{xx},w\rangle
=-\int_\TT f_{xx}w_x w\,dx-\int_\TT f_x w_x^2\,dx,
\]
so
\[
|\langle f_x w_{xx},w\rangle|
\lesssim \big(\|f_{xx}\|_{L^\infty}+\|f_x\|_{L^\infty}\big)\,\|w\|_{H^1}^2.
\]
For the second term we have that
\[
|\langle g_{xx}w_x,w\rangle|
\le \|g_{xx}\|_{L^\infty}\,\|w_x\|_{L^2}\,\|w\|_{L^2}
\lesssim \|g_{xx}\|_{L^\infty}\,\|w\|_{H^1}^2.
\]
Hence
\[
|\langle M_2,w\rangle|\lesssim \big(\|f\|_{W^{2,\infty}}+\|g\|_{W^{2,\infty}}\big)\,\|w\|_{H^1}^2
\lesssim K(t)\,\|w\|_{H^1}^2.
\]

\medskip
For $M_3$ we decompose
\[
M_3=f f_{xxx}-g g_{xxx}=f\,w_{xxx}+w\,g_{xxx}.
\]
For the term $\langle w\,g_{xxx},w\rangle$ we integrate by parts once to avoid $g_{xxx}$:
\[
\langle w\,g_{xxx},w\rangle=\int_\TT g_{xxx}w^2\,dx
=-2\int_\TT g_{xx} w w_x\,dx,
\]
and therefore
\[
|\langle w\,g_{xxx},w\rangle|
\le 2\|g_{xx}\|_{L^\infty}\,\|w\|_{L^2}\,\|w_x\|_{L^2}
\lesssim \|g_{xx}\|_{L^\infty}\,\|w\|_{H^1}^2
\lesssim K(t)\,\|w\|_{H^1}^2.
\]
For $\langle f\,w_{xxx},w\rangle$ we integrate by parts as before:
\[
\int_\TT f\,w_{xxx}w\,dx
=-\int_\TT f_x w_{xx}w\,dx-\int_\TT f w_{xx}w_x\,dx
=\frac12\int_\TT f_{xx}w_x^2\,dx-\int_\TT f_x w_{xx}w\,dx,
\]
and integrating by parts once more in the last term,
\[
-\int_\TT f_x w_{xx}w\,dx
=\int_\TT f_{xx}w_x w\,dx+\int_\TT f_x w_x^2\,dx.
\]
Thus
\[
|\langle f\,w_{xxx},w\rangle|
\lesssim \big(\|f_x\|_{L^\infty}+\|f_{xx}\|_{L^\infty}\big)\,\|w\|_{H^1}^2
\lesssim K(t)\,\|w\|_{H^1}^2,
\]
and altogether $|\langle M_3,w\rangle|\lesssim K(t)\|w\|_{H^1}^2$.

\medskip
Finally, for $M_4$ we have
\[
M_4=f f_x-g g_x=f w_x+w g_x,
\]
hence
\[
\langle f w_x,w\rangle=-\frac12\int_\TT f_x w^2\,dx,
\qquad
\langle w g_x,w\rangle=\int_\TT g_x w^2\,dx,
\]
and thus
\[
|\langle M_4,w\rangle|
\lesssim \big(\|f_x\|_{L^\infty}+\|g_x\|_{L^\infty}\big)\,\|w\|_{L^2}^2
\lesssim K(t)\,\|w\|_{H^1}^2.
\]

\medskip
Collecting the previous bounds we obtain
\[
\big|\langle \mathcal{F}_{\mathrm{nl}}(f)-\mathcal{F}_{\mathrm{nl}}(g),\,w\rangle\big|
\lesssim K(t)\,\|w\|_{H^1}^2.
\]
Together with the corresponding linear estimate, this yields
\[
\frac{d}{dt}E_w(t)\le C\,K(t)\,E_w(t),
\qquad E_w(t)\simeq \|w(t)\|_{H^1}^2,
\]
and since $w(0)=0$, Gr\"onwall's inequality implies $E_w(t)\equiv 0$ on $[0,T]$. Hence $f\equiv g$ and the solution is unique.

\end{proof}

\noindent

Next we establish a continuation/breakdown criterion for the strong solution.  Since most of the argument reuses the energy estimates from the local existence theorem, we only provide a sketch of the proof and highlight the most relevant changes.

\begin{theorem}\label{th:continuation}
	Let $s>\frac{5}{2}$ and let $f_0\in H^s(\TT)$ be mean-zero initial data for \eqref{asymptotic:model:contr}.
	Let $f\in C([0,T);H^s(\TT))$ be the (unique) strong solution given by Theorem \ref{th:existence:strong} on its
	maximal interval of existence $[0,T)$ (so that $T\in(0,\infty]$ is maximal in $H^s$).
	Then the following hold:
	\begin{enumerate}
		\item[(i)] (\emph{Continuation}) If for some $0<T^\sharp<T$ we have
		\begin{equation}\label{crit:intLip}
			\int_0^{T^\sharp}\|f_x(t)\|_{L^\infty(\TT)}\,dt<\infty,
		\end{equation}
		then $\sup_{t\in[0,T^\sharp]}\|f(t)\|_{H^s(\TT)}<\infty$, and in particular the solution extends
		beyond $T^\sharp$ as a strong solution in $H^s(\TT)$.
		
		\item[(ii)] (\emph{Breakdown}) If $T<\infty$, then necessarily
		\begin{equation}\label{crit:breakdown}
			\int_0^{T}\|f_x(t)\|_{L^\infty(\TT)}\,dt=+\infty.
		\end{equation}
	\end{enumerate}
\end{theorem}
\begin{proof}[Proof of Theorem \ref{th:continuation}]
We keep the notation of the proof of Theorem \ref{th:existence:strong}. In particular,
$\mathcal{P},\mathcal{M}$ are as in \eqref{opP}--\eqref{multiplierP-multiplierM}, $\mathcal{M}=\Id+\mathcal{S}$,
and we use the energy
\[
\mathcal{E}(t):=\|f(t)\|_{L^2(\TT)}^{2}+\|\Lambda^s f(t)\|_{L^2(\TT)}^{2}\sim \|f(t)\|_{H^s(\TT)}^{2}.
\]
We also recall that the mean is conserved (Step 1 of Theorem \ref{th:existence:strong}), hence $f$ remains mean-zero.

\medskip\noindent
\underline{\textit{Step 1: a refined a priori inequality.}}
We claim that for $s>\frac52$ the energy satisfies
\begin{equation}\label{energy:tame}
	\frac{d}{dt}\mathcal{E}(t)\le C_0\,\mathcal{E}(t)+C_1\,\|f_x(t)\|_{L^\infty(\TT)}\,\mathcal{E}(t),
\end{equation}
where $C_0,C_1>0$ depend only on the parameters of the model and on $s$.

\smallskip
Testing \eqref{asymptotic:model:contr2} against $f$ yields
$\frac12\frac{d}{dt}\|f\|_2^2=I_1+I_2$ with $I_1,I_2$ as in the proof of Theorem \ref{th:existence:strong}.
The linear term satisfies
\begin{equation}\label{cont:I1}
	|I_1|\lesssim \|f\|_2^2+\|f_x\|_2^2\lesssim \mathcal{E}(t),
\end{equation}
by Lemma \ref{lemma:op}. For $I_2$, we use $\mathcal{M}=\Id+\mathcal{S}$ and estimate as in the local existence
proof, keeping $\|f_x\|_{L^\infty}$ explicit (and using that $\mathcal{S}$ is lower order). This gives
\begin{equation}\label{cont:L2}
	\frac{d}{dt}\|f(t)\|_{2}^{2}\le C\,\mathcal{E}(t)+C\,\|f_x(t)\|_{L^\infty}\,\mathcal{E}(t).
\end{equation}

\smallskip
Applying $\Lambda^s$ to \eqref{asymptotic:model:contr2} and testing against $\Lambda^s f$
gives $\frac12\frac{d}{dt}\|\Lambda^s f\|_2^2=J_1+J_2$ with $J_1,J_2$ as in the proof of
Theorem \ref{th:existence:strong}. The linear contribution is unchanged:
\begin{equation}\label{cont:J1}
	|J_1|\lesssim \|\Lambda^s f\|_2^2\lesssim \mathcal{E}(t),
\end{equation}
by \eqref{estimate:J1}.

For $J_2$, decompose again $\mathcal{M}=\Id+\mathcal{S}$. The terms containing $\mathcal{S}$ are lower order and
are bounded by $C(1+\|f_x\|_\infty)\mathcal{E}(t)$ using Lemma \ref{lemma:op}. For the principal part (with
$\mathcal{M}$ replaced by $\Id$), one has
\begin{equation}\label{J21J22J24:tame}
|J_{21}|+|J_{22}|+|J_{24}|\;\lesssim\;\Big(1+\|f_x\|_{L^\infty(\TT)}\Big)\,\|\Lambda^s f\|_2^2.
\end{equation}
For the most singular term $J_{23}$, arguing as in Theorem \ref{th:existence:strong}, we write
\[
J_{23}=-\frac{3\nu}{8}\int_{\TT} f_{x}^{2}\,\Lambda^{2s}\mathcal{P}f_{x}\,dx
-\frac{\nu}{4}\int_{\TT} ff_{x}\,\Lambda^{2s}\mathcal{P}f_{xx}\,dx.
\]
The first integral is treated as in $J_{22}$. For the second one, we use the identity
$\mathcal{P}f_{xx}=\frac{\nu}{2}\partial_{xx}\mathcal{P}f=-f+\kappa \mathcal{P}f$ (Lemma \ref{lemma:op}) to obtain
\[
-\frac{\nu}{4}\int_{\TT} ff_x\,\Lambda^{2s}\mathcal{P}f_{xx}\,dx
= \frac12\int_{\TT} ff_x\,\Lambda^{2s}f\,dx
-\frac{\kappa}{2}\int_{\TT} ff_x\,\Lambda^{2s}\mathcal{P}f\,dx.
\]
The last term is handled as $J_{24}$. For the first one we use
\[
\int_{\TT} ff_x\,\Lambda^{2s}f\,dx
=-\frac12\int_{\TT} f_x|\Lambda^s f|^2\,dx
+\int_{\TT} [\Lambda^s,f]f_x\,\Lambda^s f\,dx,
\]
and thus, by Hölder and \eqref{KPcom},
\[
\Big|\int_{\TT} ff_x\,\Lambda^{2s}f\,dx\Big|
\lesssim \|f_x\|_{L^\infty}\|\Lambda^s f\|_2^2
+\|[\Lambda^s,f]f_x\|_2\,\|\Lambda^s f\|_2
\lesssim \|f_x\|_{L^\infty}\|\Lambda^s f\|_2^2.
\]
Collecting the previous bounds yields
\begin{equation}\label{J2:tame}
|J_2|\lesssim \Big(1+\|f_x(t)\|_{L^\infty(\TT)}\Big)\,\|\Lambda^s f(t)\|_2^2
\lesssim \Big(1+\|f_x(t)\|_{L^\infty(\TT)}\Big)\,\mathcal{E}(t).
\end{equation}
Combining \eqref{cont:J1} and \eqref{J2:tame}, we obtain
\begin{equation}\label{Hs:tame}
\frac{d}{dt}\|\Lambda^s f(t)\|_2^2
\le C\Big(1+\|f_x(t)\|_{L^\infty(\TT)}\Big)\,\mathcal{E}(t).
\end{equation}
Finally, adding \eqref{cont:L2} and \eqref{Hs:tame} gives \eqref{energy:tame}.

\medskip
Integrating \eqref{energy:tame} and applying Grönwall yields
\begin{equation}\label{gronwall:tame}
	\mathcal{E}(t)\le \mathcal{E}(0)\,
	\exp\!\Big(C_0 t + C_1\int_0^{t}\|f_x(\sigma)\|_{L^\infty(\TT)}\,d\sigma\Big),
	\qquad 0\le t<T.
\end{equation}
In particular, if \eqref{crit:intLip} holds for some $T^\sharp<T$, then
$\sup_{t\in[0,T^\sharp]}\|f(t)\|_{H^s(\TT)}<\infty$.

\medskip\noindent
\underline{\textit{Step 2: proving (i).}}
Fix $T^\sharp<T$ satisfying \eqref{crit:intLip}. Then \eqref{gronwall:tame} implies
$\|f(T^\sharp)\|_{H^s(\TT)}\le C^\sharp<\infty$.
Restarting Theorem \ref{th:existence:strong} at time $T^\sharp$ with datum $f(T^\sharp)$ yields an extension of $f$
beyond $T^\sharp$.

\medskip\noindent
\underline{\textit{Step 3: proving (ii).}}
If $T<\infty$ and \eqref{crit:breakdown} fails, then there exists $T^\sharp<T$ arbitrarily close to $T$ such that
\eqref{crit:intLip} holds. By (i) the solution extends beyond $T^\sharp$, contradicting the maximality of $T$.
Therefore \eqref{crit:breakdown} must hold.
\end{proof}

\section{Global existence and decay in the BBM regime}\label{sec:bbm-global}
In this section we establish the global existence classical solutions and their decay in time for the asymptotic model in the BBM-type regime (corresponding to the Benjamin–Bona–Mahony equation setting, $\nu=0$). In this case, the system becomes a regularized hyperbolic equation of dispersive-dissipative type.
\begin{theorem}\label{th:bbm-global}
	Assume $\nu=0$ and $\beta>-2$. Let $f_0\in H^2(\TT)$ be mean-zero and let
	$f\in C([0,T_{\max});H^2(\TT))$ be the maximal strong solution of \eqref{asymptotic:model:contr} with $f(0)=f_0$. Then there exists $\delta_0=\delta_0(\varepsilon,\kappa,\beta)>0$ such that if
	\[
	\|f_0\|_{H^2}\le \delta_0,
	\]
	the solution is global, $T_{\max}=+\infty$, and there exist constants $C,c>0$, depending only on
	$(\varepsilon,\kappa,\beta)$, such that
	\[
	\|f(t)\|_{H^2}\le C e^{-ct}\|f_0\|_{H^2}\qquad\text{for all }t\ge 0.
	\]
	In particular, $\|f(t)\|_{H^2}\to 0$ as $t\to\infty$.
\end{theorem}

\begin{proof}[Proof of Theorem \ref{th:bbm-global}]
	We only establish the a prior energy estimates leading to global control and decay under the smallness
	assumption on $\|f_0\|_{H^2}$. The construction of solutions via the approximation/compactness procedure follow by the same standard steps used in the local well-posedness theorem
	(Theorem~\ref{th:existence:strong}). To avoid repetition we omit these routine details. 
	\medskip
	
	First notice that when $\nu=0$, the operators $\mathcal{P}$ and $\mathcal{M}$ reduce to
	\[
	\mathcal{P}=\frac{1}{\kappa}\,\mathrm{Id},\qquad
	\mathcal{M}=\Big(\mathrm{Id}-\frac{4}{\kappa^2}\partial_{xx}\Big)^{-1}\Big(\mathrm{Id}+\frac{2}{\kappa}\partial_x\Big),
	\]
	and therefore \eqref{asymptotic:model:contr} can be written in the local BBM-type form
	\begin{equation}\label{eq:bbm:general}
		f_t-\frac{4}{\kappa^2}\partial_{xx}f_t
		=\Big(1+\frac{2}{\kappa}\partial_x\Big)\Big[
		-\frac{1}{\varepsilon}\Big(1-\frac{\beta}{2}\Big)f_{xx}
		+\frac{\kappa}{\varepsilon}f_x
		+\Big(2+\frac{\beta}{4}\Big)(ff_x)_x
		-2\kappa ff_x\Big].
	\end{equation}
	We set
	\[
	a:=\frac{1}{\varepsilon}\Big(\frac{\beta}{2}-1\Big),\qquad
	b:=\frac{\kappa}{\varepsilon},\qquad
	c:=2+\frac{\beta}{4},\qquad
	d:=2\kappa,
	\]
	so that the bracket in \eqref{eq:bbm:general} is $Q:=a f_{xx}+b f_x+c(ff_x)_x-dff_x$. \medskip
	
	We introduce the following energies
	\[
	E_1(t):=\|f(t)\|_{L^2}^2+\frac{4}{\kappa^2}\|f_x(t)\|_{L^2}^2,
	\qquad
	E_2(t):=\|f_x(t)\|_{L^2}^2+\frac{4}{\kappa^2}\|f_{xx}(t)\|_{L^2}^2,
	\]
	and the dissipations $D_1(t):=\|f_x(t)\|_{L^2}^2$, $D_2(t):=\|f_{xx}(t)\|_{L^2}^2$. Note that $E_1\simeq\|f\|_{H^1}^2$ and
	$E_2\simeq\|f\|_{H^2}^2$ (with constants depending only on $\kappa$). Since $f$ has zero mean, Poincar\'e inequality yields
	\begin{equation}\label{eq:coercive-bbm}
		E_1(t)+E_2(t)\ \lesssim\ D_1(t)+D_2(t).
	\end{equation}
	We will also use the one-dimensional Gagliardo--Nirenberg bounds
	\begin{equation}\label{eq:GN-BBM}
		\|f\|_{L^\infty}\lesssim \|f\|_{L^2}^{1/2}\|f_x\|_{L^2}^{1/2}\lesssim \sqrt{E_1},
		\qquad
		\|f_x\|_{L^\infty}\lesssim \|f_x\|_{L^2}^{1/2}\|f_{xx}\|_{L^2}^{1/2}\lesssim \sqrt{E_2}.
	\end{equation}
	
	Taking the $L^2$ inner product of \eqref{eq:bbm:general} with $f$ and integrating by parts gives
	\[
	\frac12\frac{d}{dt}\|f\|_{L^2}^2-\frac{4}{\kappa^2}\int_\TT f_{xxt}f\,dx
	=\frac12\frac{d}{dt}\|f\|_{L^2}^2+\frac{2}{\kappa^2}\frac{d}{dt}\|f_x\|_{L^2}^2
	=\frac12\frac{d}{dt}E_1(t),
	\]
	while on the right-hand side
	\[
	\int_\TT\Big(1+\frac{2}{\kappa}\partial_x\Big)Q\,f\,dx
	=\int_\TT Q f\,dx-\frac{2}{\kappa}\int_\TT Q f_x\,dx.
	\]
	The linear contributions are
	\[
	\int_\TT a f_{xx}f\,dx=-aD_1,\qquad
	-\frac{2}{\kappa}\int_\TT b f_x f_x\,dx=-\frac{2b}{\kappa}D_1,
	\]
	whereas $\int_\TT b f_x f\,dx=0$ and $-\frac{2}{\kappa}\int_\TT a f_{xx} f_x\,dx=0$. Thus the linear part yields
	$-(a+2b/\kappa)D_1$. For the nonlinear part, using
	\[
	\int_\TT (ff_x)_x f\,dx=-\int_\TT f\,f_x^2\,dx,\qquad
	\int_\TT (ff_x)_x f_x\,dx=\frac12\int_\TT f_x^3\,dx,\qquad
	\int_\TT f f_x f_x\,dx=\int_\TT f\,f_x^2\,dx,
	\]
	we obtain
	\[
	\Big|\int_\TT (c(ff_x)_x-dff_x)f\,dx\Big|\lesssim \|f\|_{L^\infty}D_1,\qquad
	\Big|\int_\TT (c(ff_x)_x-dff_x)f_x\,dx\Big|\lesssim \|f_x\|_{L^\infty}D_1.
	\]
	Invoking \eqref{eq:GN-BBM} we conclude that
	\begin{equation}\label{ineq:E1}
		\frac12\frac{d}{dt}E_1(t)+\Big(a+\frac{2b}{\kappa}\Big)D_1(t)
		\le C\big(\sqrt{E_1(t)}+\sqrt{E_2(t)}\big)D_1(t),
	\end{equation}
	for some $C=C(\varepsilon,\kappa,\beta)$. \medskip
		
	Next, taking the $L^2$ inner product of \eqref{eq:bbm:general} with $-f_{xx}$ and integrating by parts yields
	\[
	\frac12\frac{d}{dt}\|f_x\|_{L^2}^2-\frac{4}{\kappa^2}\int_\TT f_{xxt}(-f_{xx})\,dx
	=\frac12\frac{d}{dt}\|f_x\|_{L^2}^2+\frac{2}{\kappa^2}\frac{d}{dt}\|f_{xx}\|_{L^2}^2
	=\frac12\frac{d}{dt}E_2(t),
	\]
	and
	\[
	\int_\TT\Big(1+\frac{2}{\kappa}\partial_x\Big)Q\,(-f_{xx})\,dx
	=\int_\TT Q(-f_{xx})\,dx-\frac{2}{\kappa}\int_\TT Q(-f_{xxx})\,dx.
	\]
	The linear terms satisfy
	\[
	\int_\TT a f_{xx}(-f_{xx})\,dx=-aD_2,\qquad
	-\frac{2}{\kappa}\int_\TT b f_x(-f_{xxx})\,dx
	=\frac{2b}{\kappa}\int_\TT f_x f_{xxx}\,dx=-\frac{2b}{\kappa}D_2,
	\]
	while $\int_\TT b f_x(-f_{xx})\,dx=0$ and $-\frac{2}{\kappa}\int_\TT a f_{xx}(-f_{xxx})\,dx=0$, hence the linear part yields
	$-(a+2b/\kappa)D_2$. For the nonlinear terms we use we find that
	\[
	\Big|\int_\TT (ff_x)_x f_{xx}\,dx\Big|
	\lesssim \|f_x\|_{L^\infty}\|f_x\|_{L^2}\|f_{xx}\|_{L^2}+\|f\|_{L^\infty}\|f_{xx}\|_{L^2}^2
	\lesssim (\|f\|_{L^\infty}+\|f_x\|_{L^\infty})D_2,
	\]
	and similarly
	\[
	\Big|\int_\TT (ff_x)_x f_{xxx}\,dx\Big|
	+\Big|\int_\TT (f f_x) f_{xxx}\,dx\Big|
	\lesssim (\|f\|_{L^\infty}+\|f_x\|_{L^\infty})D_2.
	\]
	Using \eqref{eq:GN-BBM} and Young's inequality to absorb mixed products we obtain
	\begin{equation}\label{ineq:E2}
		\frac12\frac{d}{dt}E_2(t)+\Big(a+\frac{2b}{\kappa}\Big)D_2(t)
		\le C\big(\sqrt{E_1(t)}+\sqrt{E_2(t)}\big)D_2(t).
	\end{equation}
	
	Setting $F:=E_1+E_2$, $G:=D_1+D_2$ and $\mu:=a+\frac{2b}{\kappa}>0$, adding \eqref{ineq:E1} and \eqref{ineq:E2} gives
	\begin{equation}\label{eq:F-ineq}
		\frac12 F'(t)+\mu\,G(t)\le C\sqrt{F(t)}\,G(t).
	\end{equation}
	Choose $\delta_0>0$ such that $C\delta_0\le \mu/2$ and assume $\|f_0\|_{H^2}\le\delta_0$ (absorbing the equivalence between
	$\|f_0\|_{H^2}$ and $F(0)^{1/2}$ into $\delta_0$). Define
	\[
	\mathcal{I}:=\Big\{t\in[0,T_{\max})\,:\,F(s)\le \delta_0^2\ \text{for all }s\in[0,t]\Big\},\qquad
	T_*:=\sup\mathcal{I}.
	\]
	Then $T_*>0$ and $F(t)\le \delta_0^2$ for $t\in[0,T_*)$. On this interval \eqref{eq:F-ineq} implies
	\[
	\frac12 F'(t)+\mu G(t)\le C\delta_0 G(t)\le \frac{\mu}{2}G(t),
	\qquad\text{hence}\qquad
	F'(t)+\mu G(t)\le 0 \quad \text{on }[0,T_*).
	\]
	In particular $F$ is non-increasing on $[0,T_*)$ and thus $F(t)\le F(0)\le \delta_0^2$ there; by continuity this forces
	$T_*=T_{\max}$, so the inequality $F'(t)+\mu G(t)\le 0$ holds for all $t\in[0,T_{\max})$.
	
	Finally, using the coercivity \eqref{eq:coercive-bbm} we have $G\gtrsim F$, hence
	\[
	F'(t)+c_1 F(t)\le 0
	\]
	for some $c_1=c_1(\varepsilon,\kappa,\beta)>0$, which yields $F(t)\le F(0)e^{-c_1 t}$. Since $F\simeq \|f\|_{H^2}^2$,
	we conclude $\|f(t)\|_{H^2}\le C e^{-ct}\|f_0\|_{H^2}$ for suitable $C,c>0$, and in particular $T_{\max}=+\infty$ and
	$\|f(t)\|_{H^2}\to 0$ as $t\to\infty$.
\end{proof}
	
\section{Numerical simulations}\label{sec:numerical}

We present numerical experiments for the asymptotic equation \eqref{asymptotic:model:contr}.
The computations are performed on the periodic domain $[0,2\pi]$, discretised with
$N=2^{14}$ uniformly spaced grid points.
After applying the Fourier transform to \eqref{asymptotic:model:contr}, we obtain a truncated system of
ODEs for the spectral coefficients $\widehat{f}_k(t)$, which we integrate by means of the Runge--Kutta method implemented in the \texttt{solve\_ivp} routine with \texttt{RK45} from the \texttt{SciPy} library, using tolerances
$\texttt{rtol}=\texttt{atol}=10^{-8}$.
The nonlinear terms are evaluated pseudospectrally using \texttt{NumPy}'s
\texttt{rfft}/\texttt{irfft} routines: products are computed in physical
space through the inverse FFT, while derivatives are implemented in Fourier space as multiplication by $ik$, $-k^2$, and $-ik^3$.
All initial data satisfy the mean-zero condition required by Theorem~\ref{th:existence:strong}. More precisely, we consider
\[
f_0(x)=A\,\operatorname{sech}^2(x-\pi)-\mu_{f_0},
\]
where $\mu_{f_0}$ denotes the spatial mean of $A\,\operatorname{sech}^2(x-\pi)$. \medskip

For $\nu>0$ and sufficiently large initial amplitude, the adaptive time integrator requires increasingly small time steps in order to satisfy the prescribed tolerances, and in some cases the computation stops before reaching the target final time.
In view of Theorem~\ref{th:continuation}, the quantity
\begin{equation}\label{eq:bkm}
    \int_0^{T'}\|f_x(t)\|_{L^\infty}\,dt
\end{equation}
plays a central role in the continuation of strong solutions: finiteness of \eqref{eq:bkm} allows continuation, whereas any finite maximal time of existence must be accompanied by its divergence.
This motivates monitoring the following two quantities in our simulations:
\begin{enumerate}
    \item[(i)] $1/\|f_x(t)\|_{L^\infty}$,
    \item[(ii)] the cumulative integral
    \[
    I(t)=\int_0^t \|f_x(s)\|_{L^\infty}\,ds,
    \]
    approximated by the trapezoidal rule.
\end{enumerate}

\subsection{Numerical experiments for $\nu>0$}

We begin with two representative simulations in the viscoelastic regime $\nu>0$, corresponding to a small-amplitude and a large-amplitude initial profile, respectively.

Figure~\ref{fig:nu1-1} shows snapshots of $f(x,t)$ for the parameter values
\begin{equation}\label{params:nu1-1}
    \nu=1,\quad \varepsilon=1,\quad \kappa=1,\quad \beta=1,\quad A=0.10.
\end{equation}
In this case, the computation reaches the final time $t=10$.

\begin{figure}[H]
    \centering
    \includegraphics[width=\linewidth]{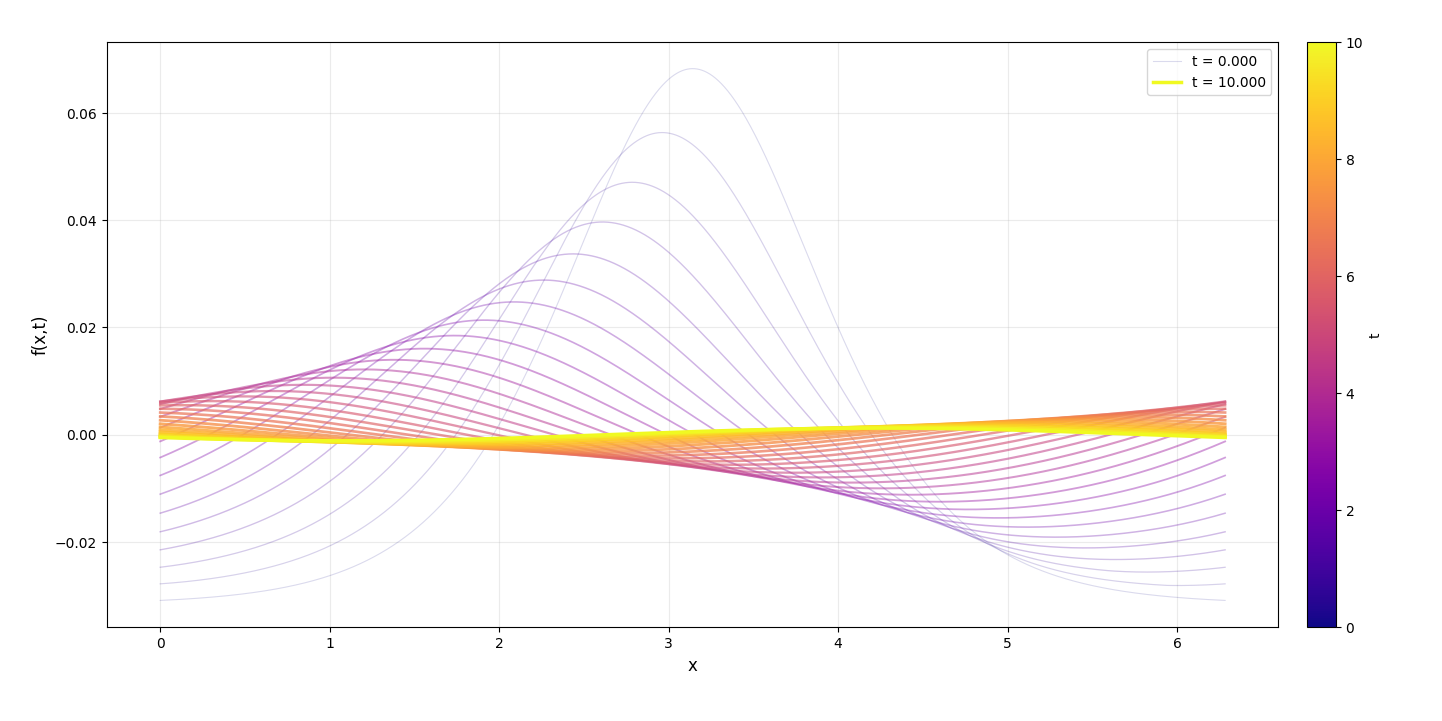}
    \caption{Evolution of $f$ for the parameter values given in \eqref{params:nu1-1}.}
    \label{fig:nu1-1}
\end{figure}

The corresponding diagnostics for $1/\|f_x\|_{L^\infty}$ and $I(t)$ are displayed in Figure~\ref{fig:nu1-5-6}.

\begin{figure}[H]
    \centering
    \includegraphics[width=0.45\linewidth]{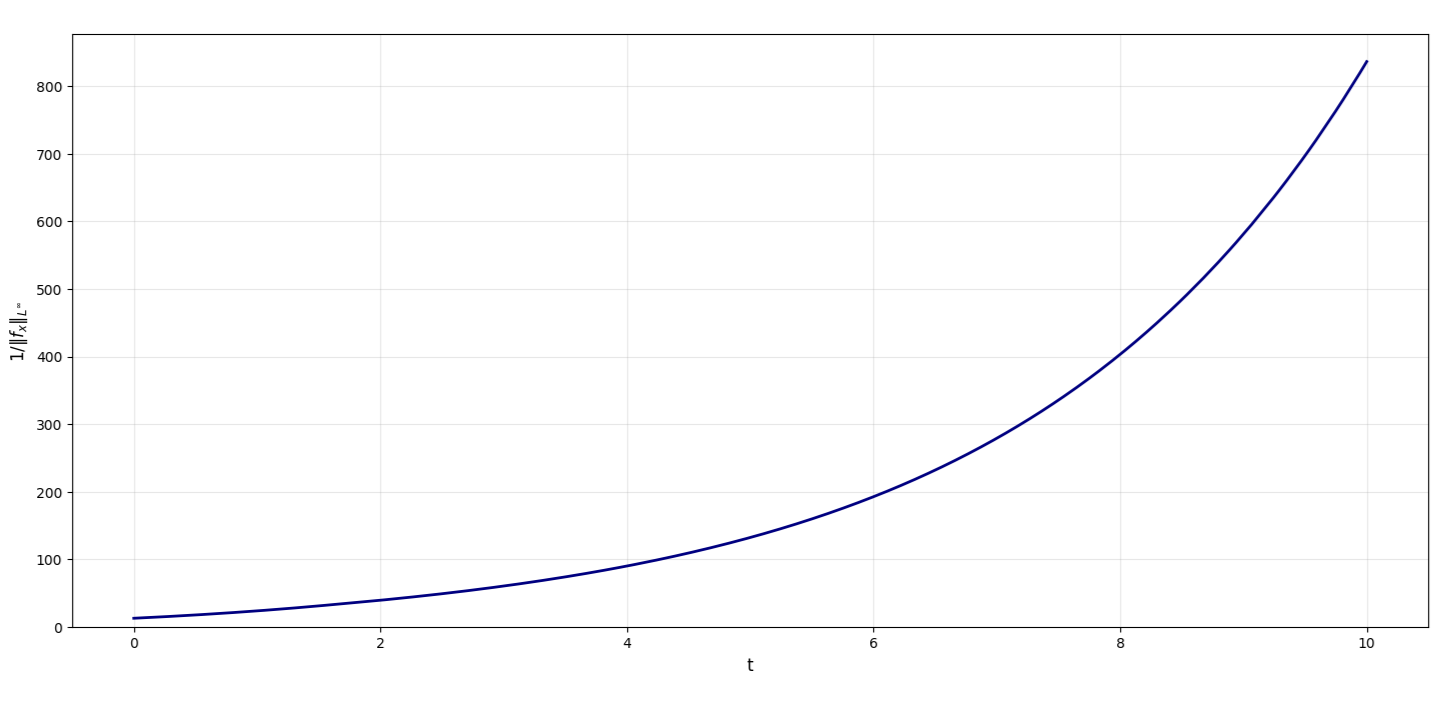}
    \includegraphics[width=0.45\linewidth]{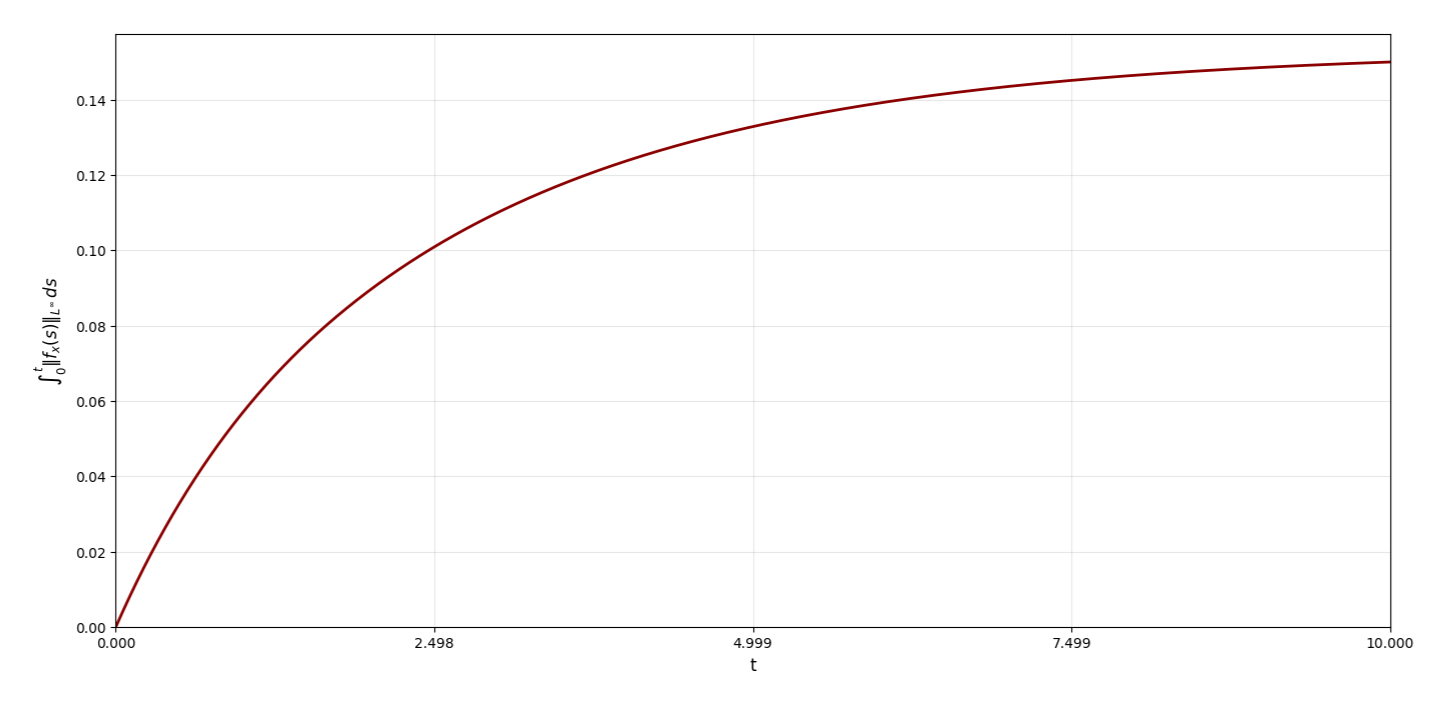}
    \caption{Diagnostics corresponding to Figure~\ref{fig:nu1-1}: left, $1/\|f_x\|_{L^\infty}$; right, $I(t)$.}
    \label{fig:nu1-5-6}
\end{figure}

Over the computed time interval, the solution remains regular and the diagnostics are consistent with a smooth dissipative evolution.

We next consider the same parameter values, except for a larger amplitude,
\begin{equation}\label{params:nu1-11}
    \nu=1,\quad \varepsilon=1,\quad \kappa=1,\quad \beta=1,\quad A=5.00.
\end{equation}
In this case, the adaptive solver stops much earlier, at approximately $t=0.665$, due to the increasingly small time step required to maintain the prescribed accuracy.

\begin{figure}[H]
    \centering
    \includegraphics[width=\linewidth]{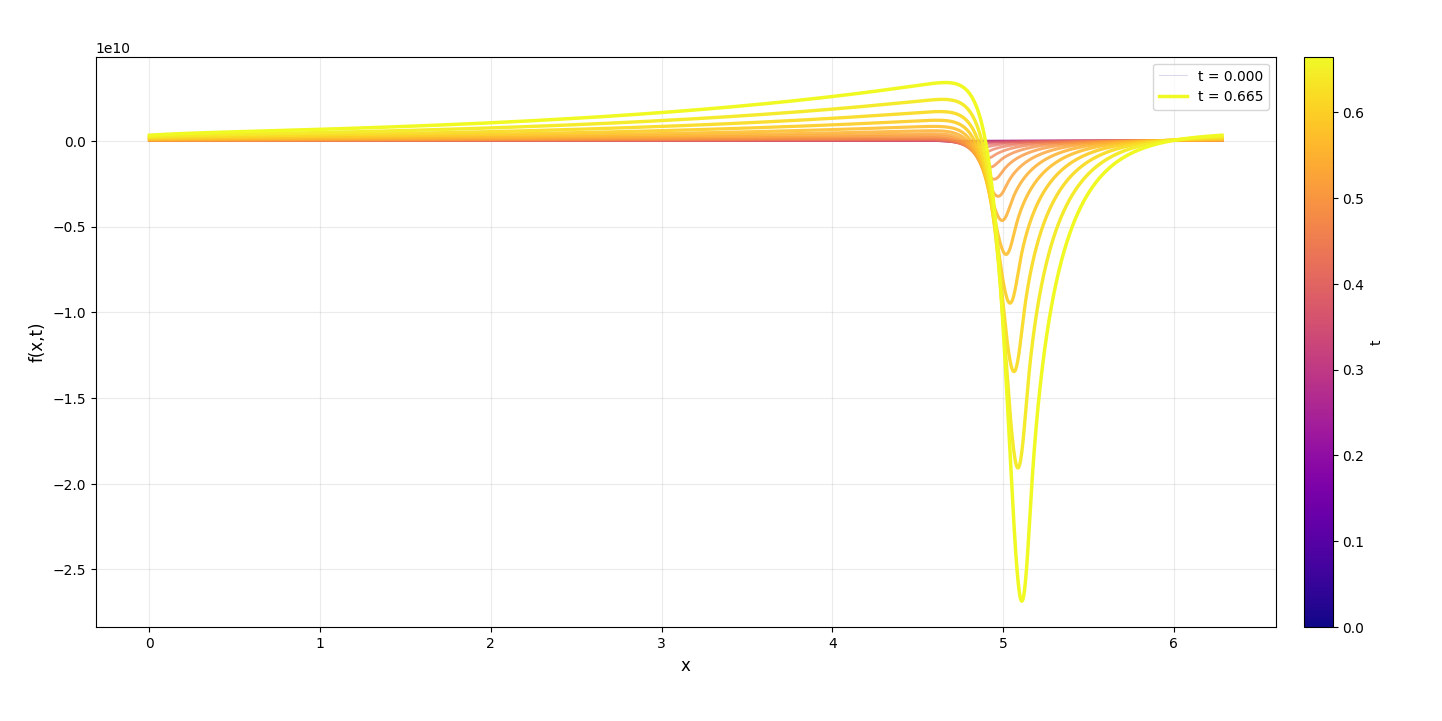}
    \caption{Evolution of $f$ for the parameter values given in \eqref{params:nu1-11}. The vertical scale reaches the order of $10^{10}$.}
    \label{fig:nu1-11}
\end{figure}

The corresponding diagnostics are shown in Figure~\ref{fig:nu1-55-66}.

\begin{figure}[H]
    \centering
    \includegraphics[width=0.45\linewidth]{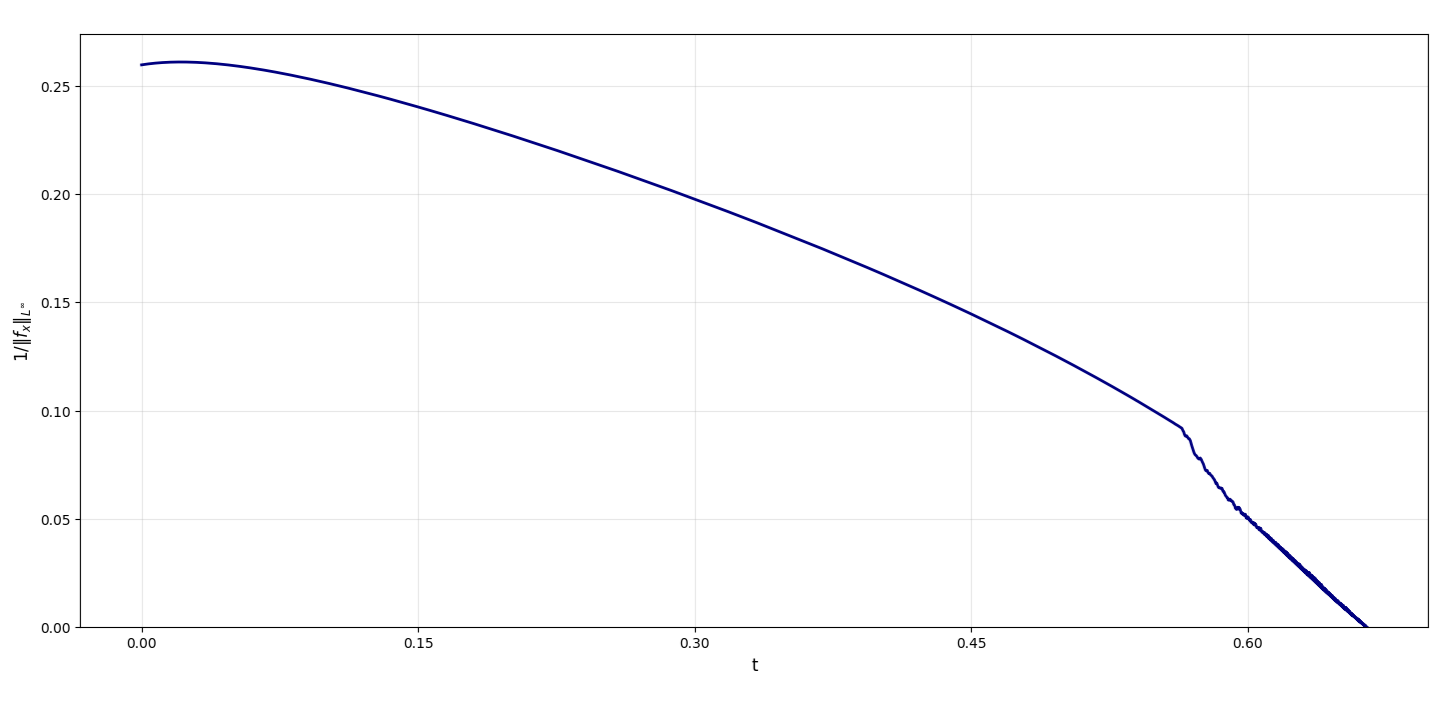}
    \includegraphics[width=0.45\linewidth]{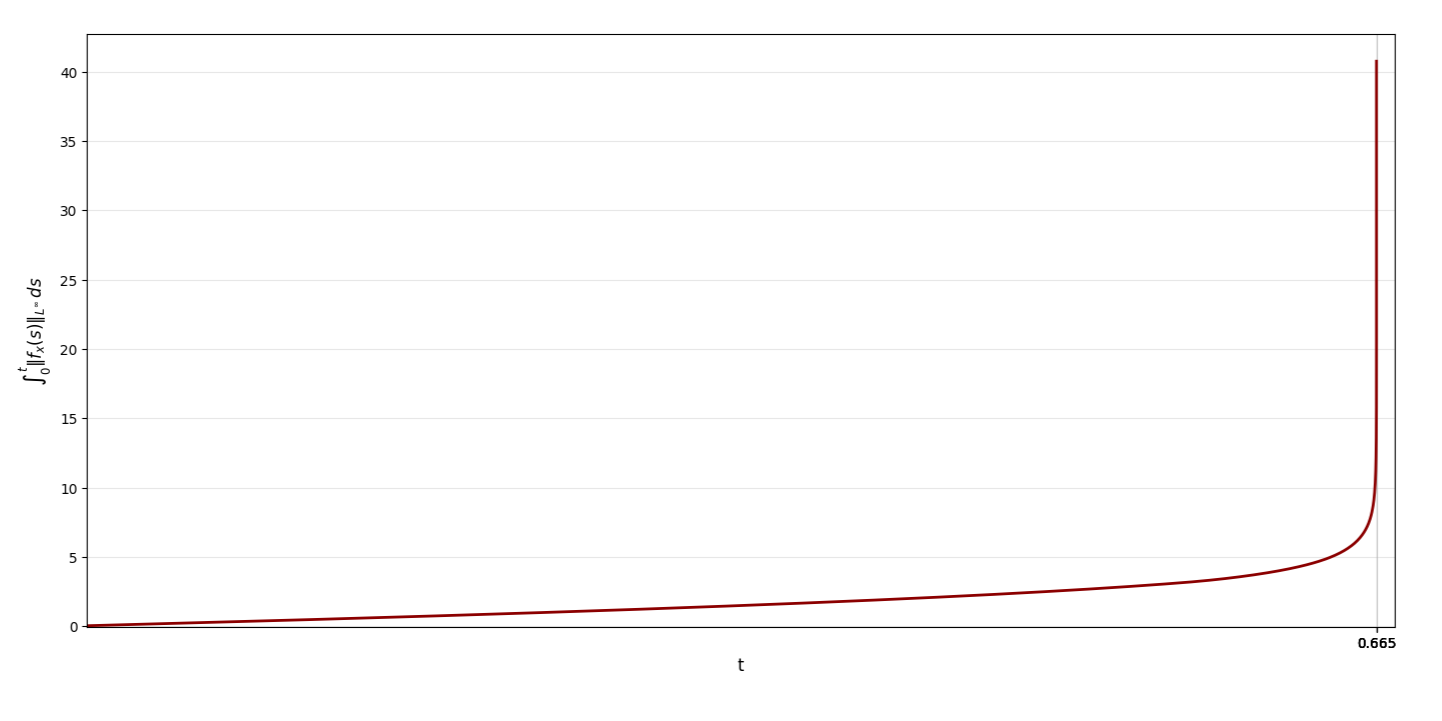}
    \caption{Diagnostics corresponding to Figure~\ref{fig:nu1-11}: left, $1/\|f_x\|_{L^\infty}$; right, $I(t)$.}
    \label{fig:nu1-55-66}
\end{figure}

Although this behavior is compatible with a possible finite-time loss of regularity, the computations do not provide conclusive numerical evidence of blow-up.
What they do indicate is a marked qualitative difference between the small-amplitude and large-amplitude regimes, as reflected in the behavior of the diagnostics in Figures~\ref{fig:nu1-5-6} and \ref{fig:nu1-55-66}.

\subsection{Numerical experiments varying $\nu$ and $A$}

We next investigate the dependence of the dynamics on the viscoelastic parameter $\nu$ and on the amplitude $A$. We first fix $A=0.1$ and vary $\nu$ over the values
\[
\nu=0,\ 0.1,\ 0.5,\ 1,\ 1.5,\ 2,\ 3,
\]
while keeping the remaining parameters equal to $1$.
Figure~\ref{fig:c1} compares the final profiles obtained in each case.

\begin{figure}[H]
    \centering
    \includegraphics[width=\linewidth]{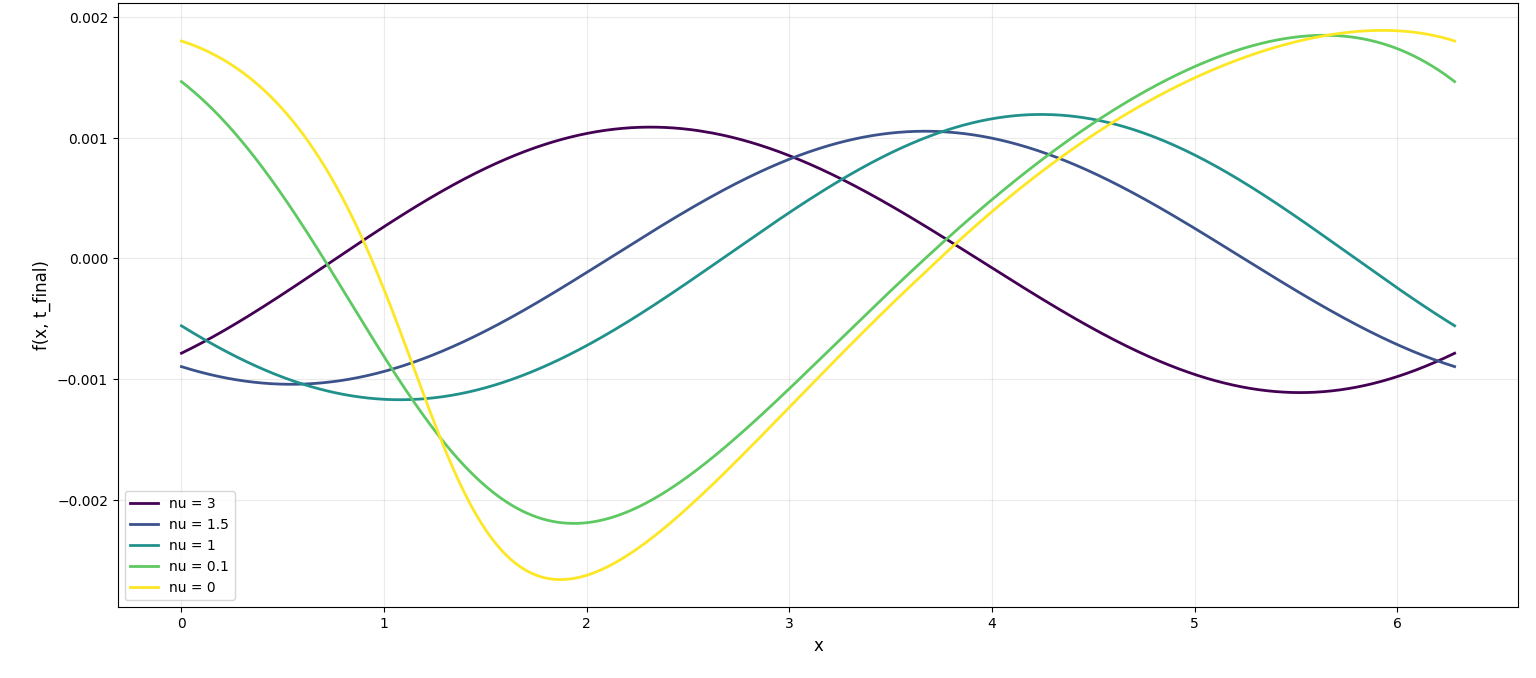}
    \caption{Comparison of the final profiles for $\nu=0, 0.1, 0.5, 1, 1.5, 2, 3$, with $A=0.1$ and all other parameters fixed to $1$.}
    \label{fig:c1}
\end{figure}

The associated diagnostics are displayed in Figure~\ref{fig:c23}.

\begin{figure}[H]
    \centering
    \includegraphics[width=0.45\linewidth]{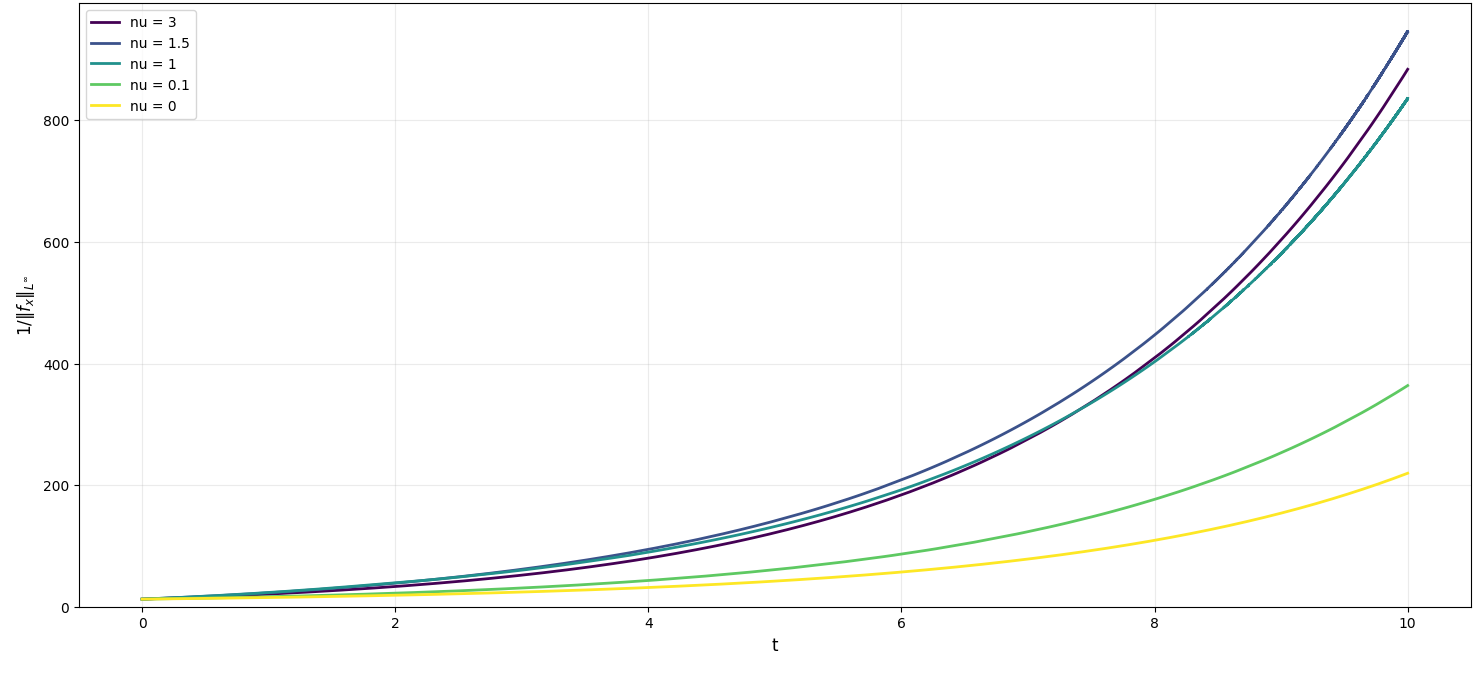}
    \includegraphics[width=0.45\linewidth]{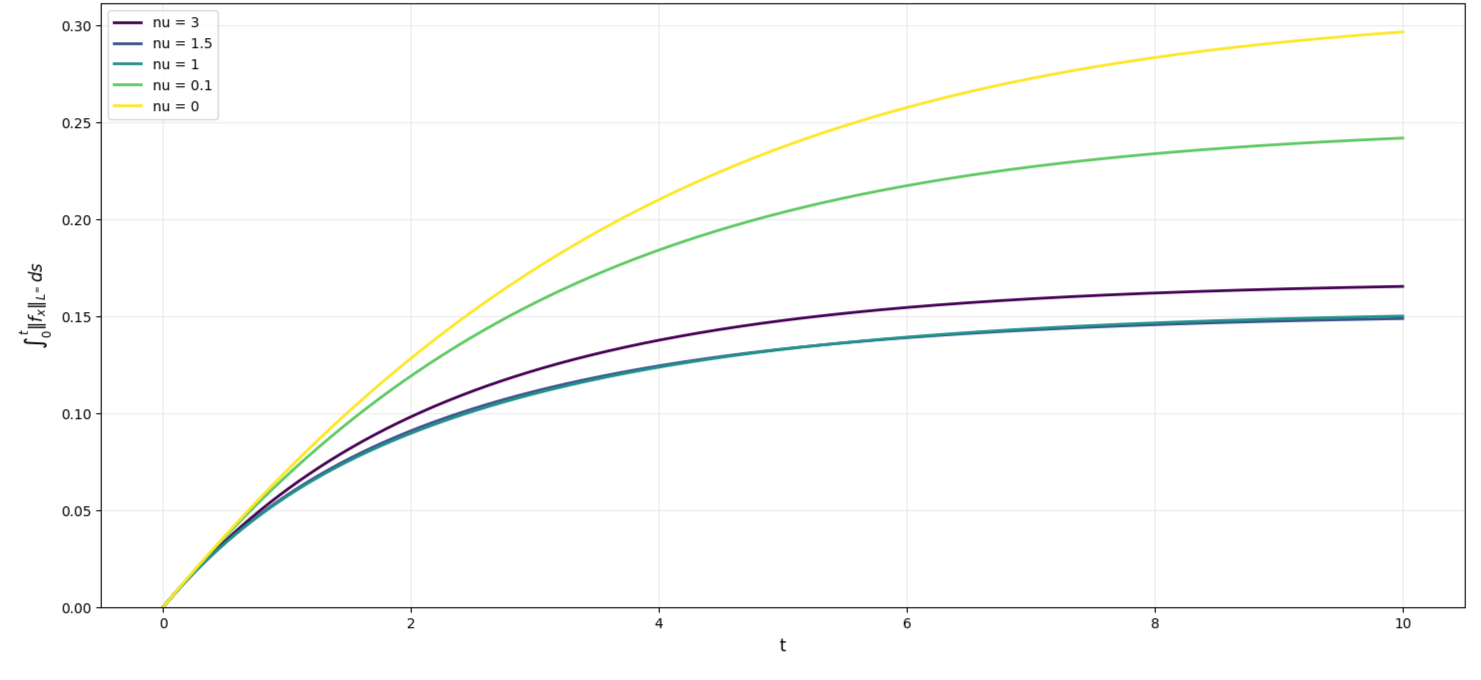}
    \caption{Diagnostics for the simulations in Figure~\ref{fig:c1}: left, $1/\|f_x\|_{L^\infty}$; right, $I(t)$.}
    \label{fig:c23}
\end{figure}

For this small-amplitude regime, the final profiles remain qualitatively similar across the tested values of $\nu$, and the diagnostics are consistent with dissipative behavior over the simulated time interval.

We next fix $\nu=1$ and vary the amplitude according to
\[
A=0.5,\ 1,\ 5,\ 10,\ 20,
\]
again keeping the remaining parameters equal to $1$.
The corresponding final profiles are shown in Figure~\ref{fig:ca1}.

\begin{figure}[H]
    \centering
    \includegraphics[width=\linewidth]{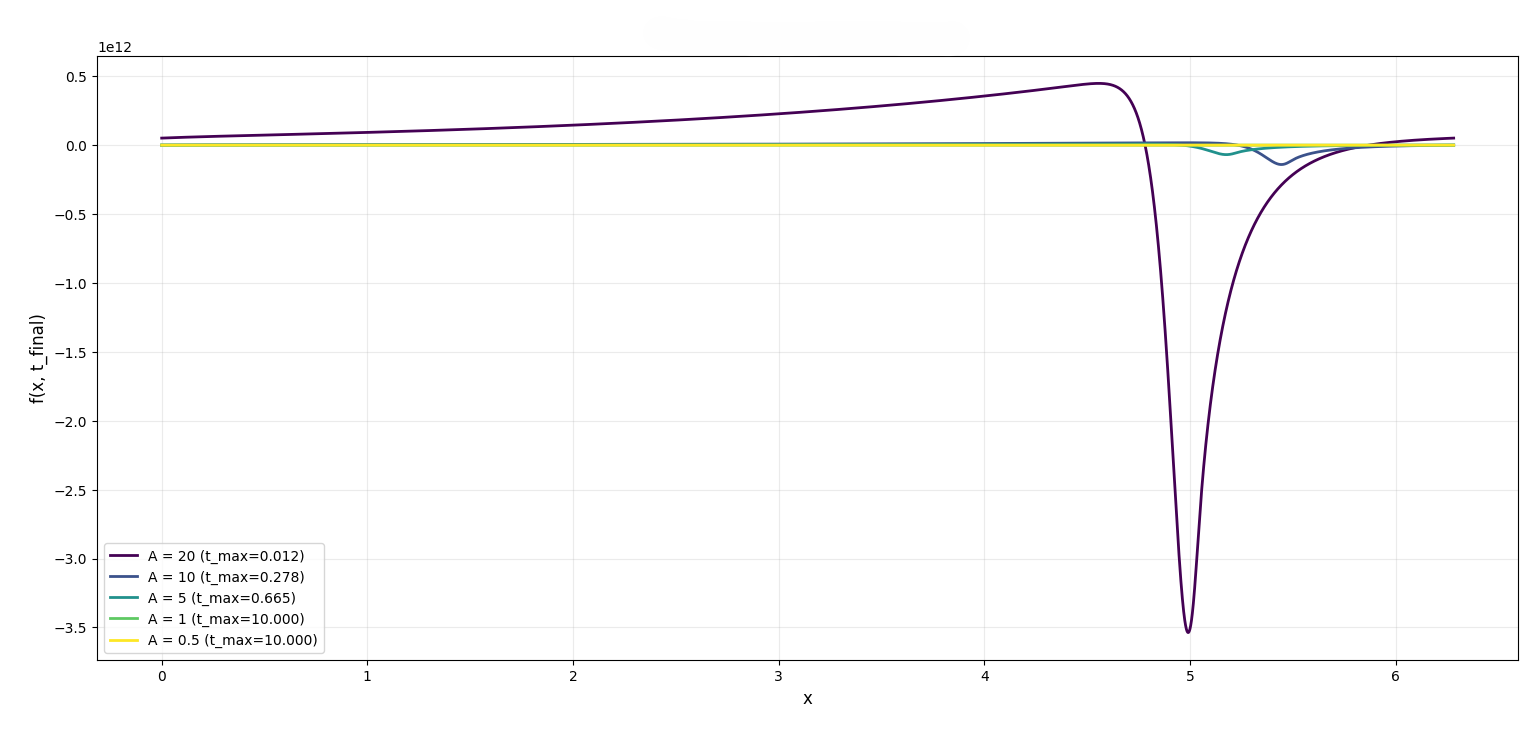}
    \caption{Comparison of the final profiles for $A=0.5, 1, 5, 10, 20$, with $\nu=1$ and all other parameters fixed to $1$. The vertical scale reaches the order of $10^{12}$.}
    \label{fig:ca1}
\end{figure}

Figure~\ref{fig:ca23} shows the corresponding diagnostics.

\begin{figure}[H]
    \centering
    \includegraphics[width=0.45\linewidth]{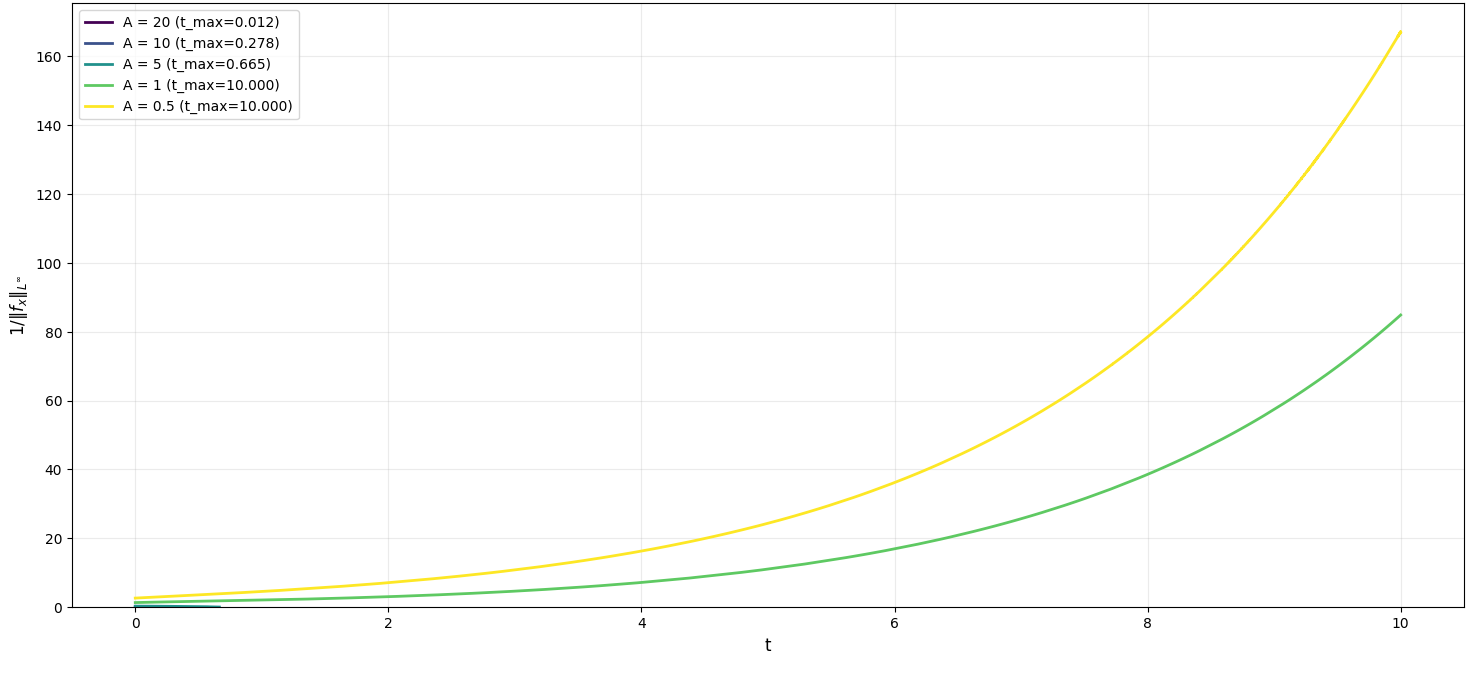}
    \includegraphics[width=0.45\linewidth]{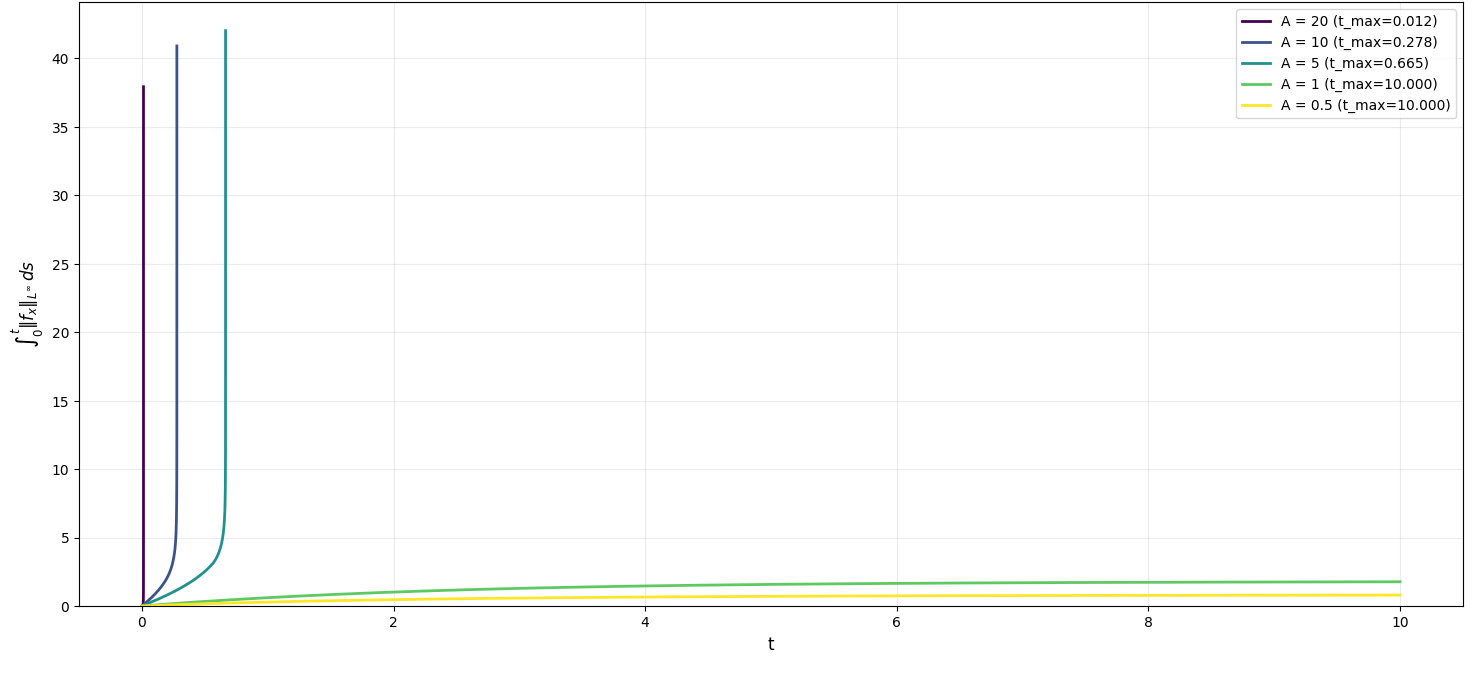}
    \caption{Diagnostics for the simulations in Figure~\ref{fig:ca1}: left, $1/\|f_x\|_{L^\infty}$; right, $I(t)$.}
    \label{fig:ca23}
\end{figure}

As the amplitude increases, the numerical behavior changes from clearly regular and dissipative regimes ($A=0.5,1$) to regimes in which gradients grow rapidly and the adaptive solver terminates much earlier.
In particular, for fixed $\nu$, larger values of $A$ lead to a substantially earlier loss of numerical resolution.

\subsection{Experiments with $\nu=0$}

We finally turn to the purely elastic regime $\nu=0$, for which the asymptotic equation reduces to a BBM-type model.
This is the regime covered by the global small-data result in Section~\ref{sec:bbm-global}.
We consider three values of $\beta$, namely
\[
\beta=2,\quad \beta=0,\quad \beta=-1,
\]
with $A=0.1$ and all remaining parameters equal to $1$. Figure~\ref{fig:beta2} presents the solution profiles for $\beta=2$.

\begin{figure}[H]
    \centering
    \includegraphics[width=\linewidth]{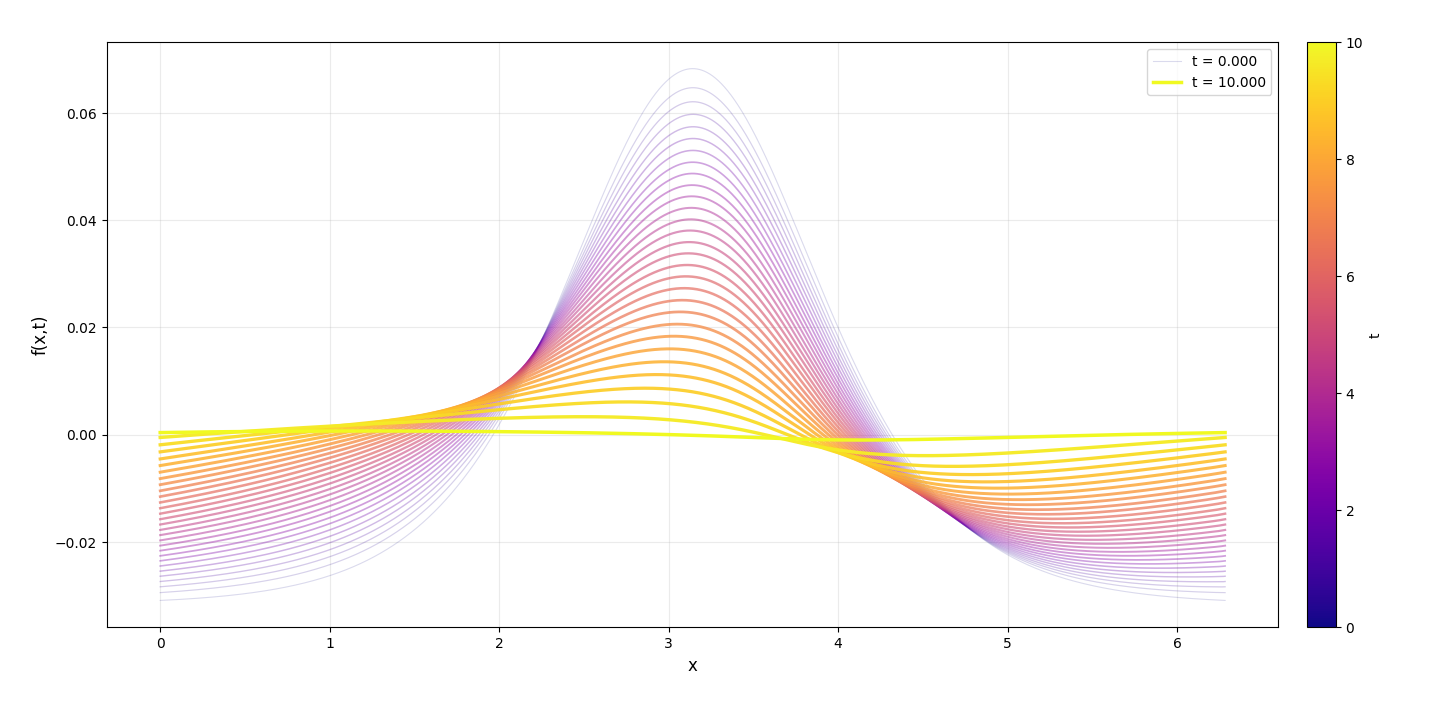}
    \caption{Numerical results for $\nu=0$ and $\beta=2$.}
    \label{fig:beta2}
\end{figure}

The corresponding diagnostics are shown in Figure~\ref{fig:beta2-2}.

\begin{figure}[H]
    \centering
    \includegraphics[width=0.45\linewidth]{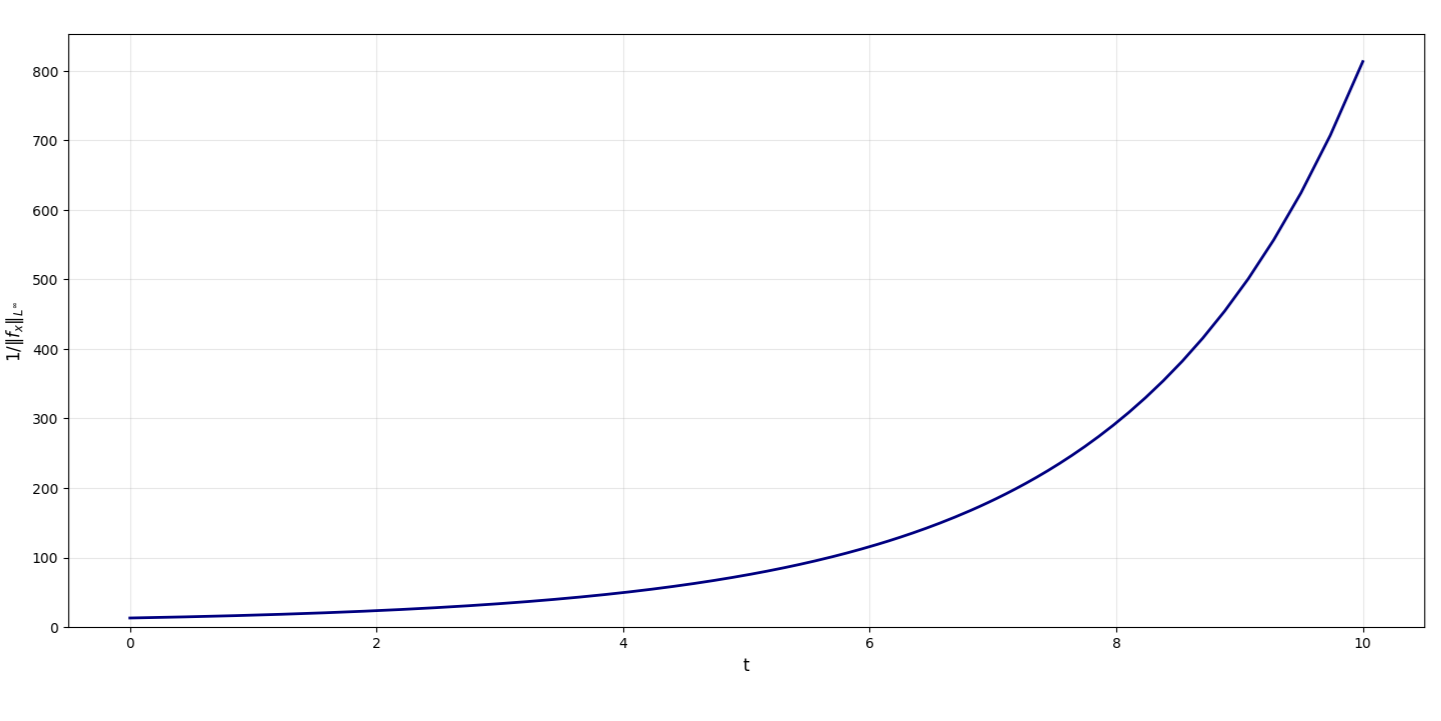}
    \includegraphics[width=0.45\linewidth]{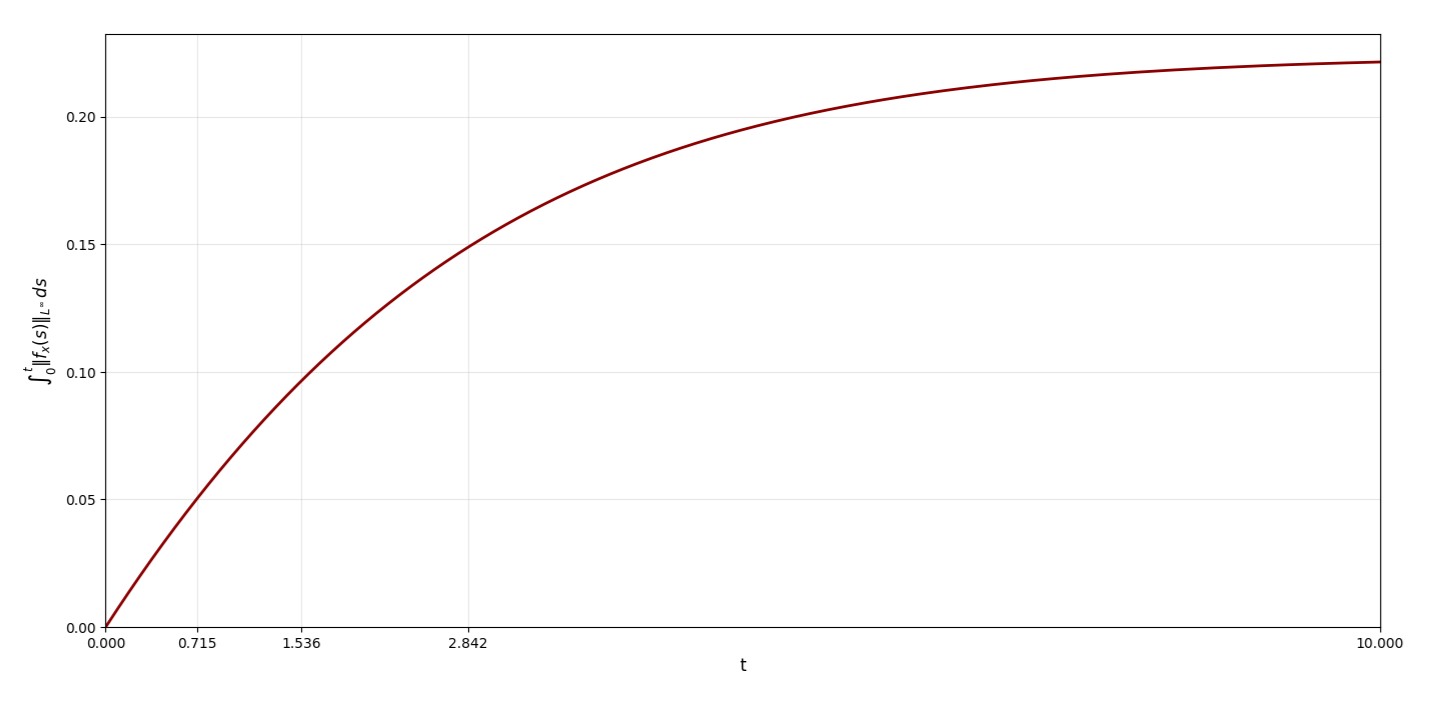}
    \caption{Diagnostics corresponding to Figure~\ref{fig:beta2}: left, $1/\|f_x\|_{L^\infty}$; right, $I(t)$.}
    \label{fig:beta2-2}
\end{figure}

Figure~\ref{fig:beta0} presents the corresponding results for $\beta=0$.

\begin{figure}[H]
    \centering
    \includegraphics[width=\linewidth]{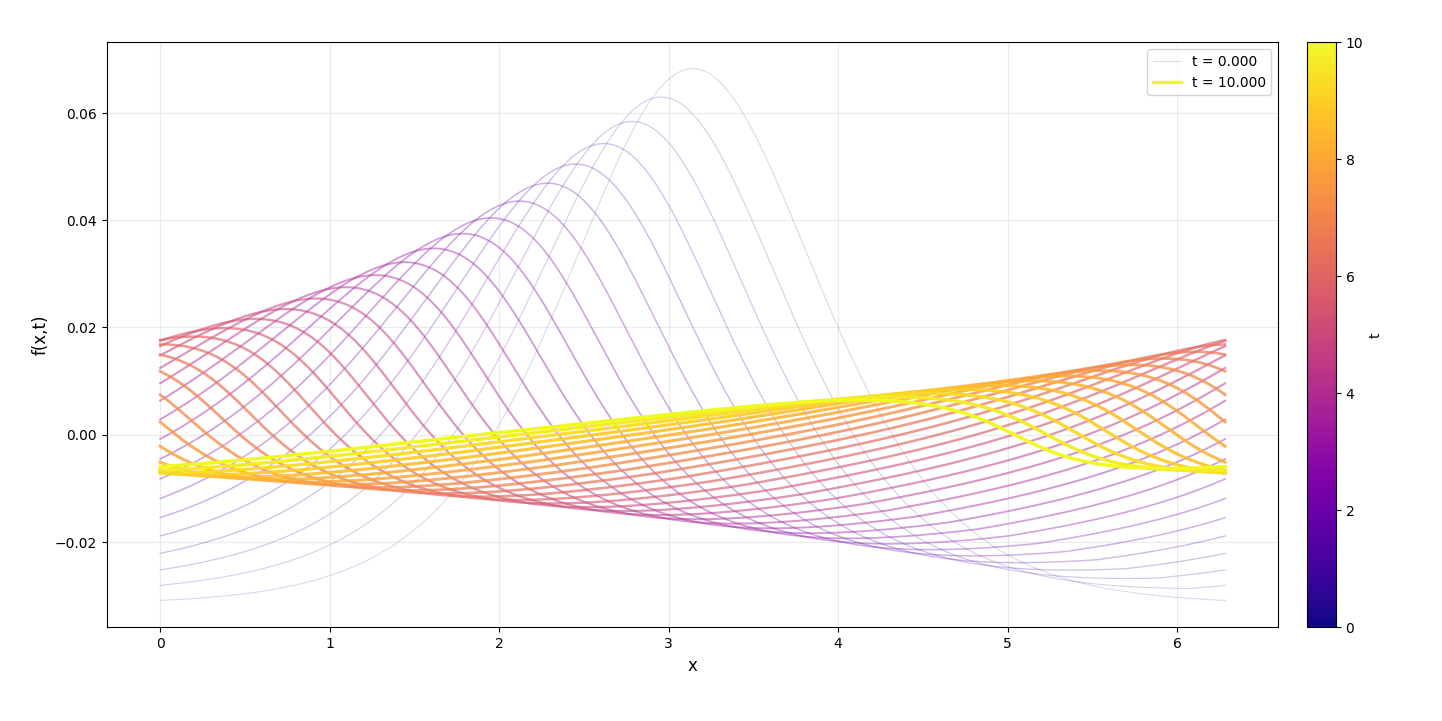}
    \caption{Numerical results for $\nu=0$ and $\beta=0$.}
    \label{fig:beta0}
\end{figure}

The associated diagnostics are displayed in Figure~\ref{fig:beta0-2}.

\begin{figure}[H]
    \centering
    \includegraphics[width=0.45\linewidth]{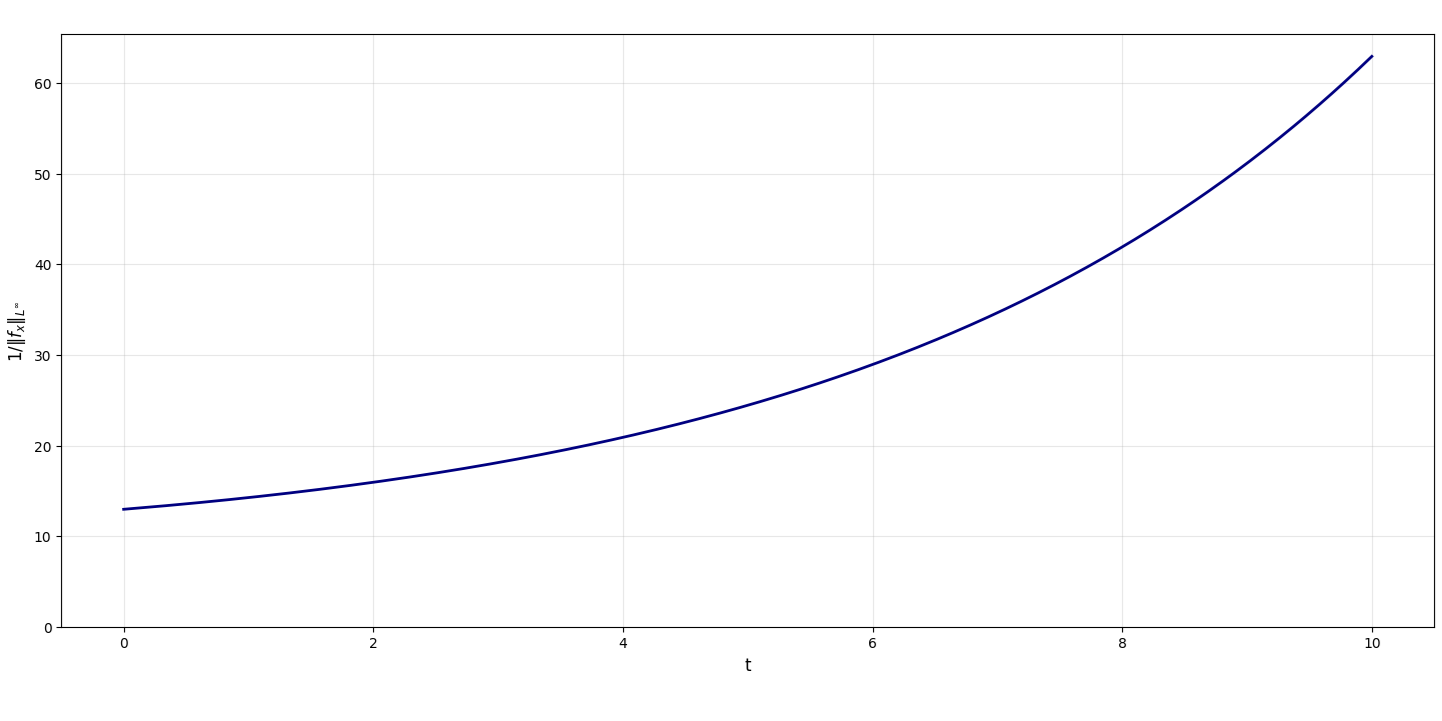}
    \includegraphics[width=0.45\linewidth]{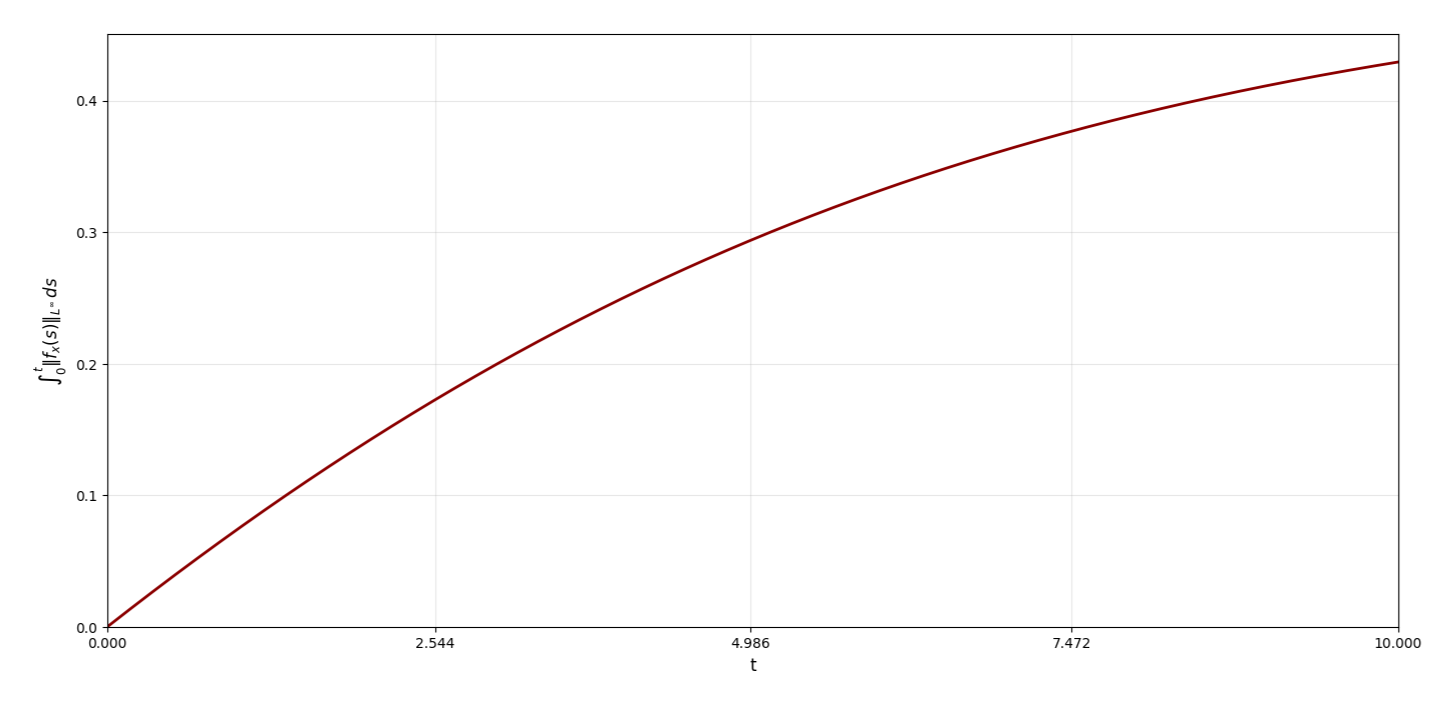}
    \caption{Diagnostics corresponding to Figure~\ref{fig:beta0}: left, $1/\|f_x\|_{L^\infty}$; right, $I(t)$.}
    \label{fig:beta0-2}
\end{figure}

Figure~\ref{fig:beta-1} presents the results for $\beta=-1$.

\begin{figure}[H]
    \centering
    \includegraphics[width=\linewidth]{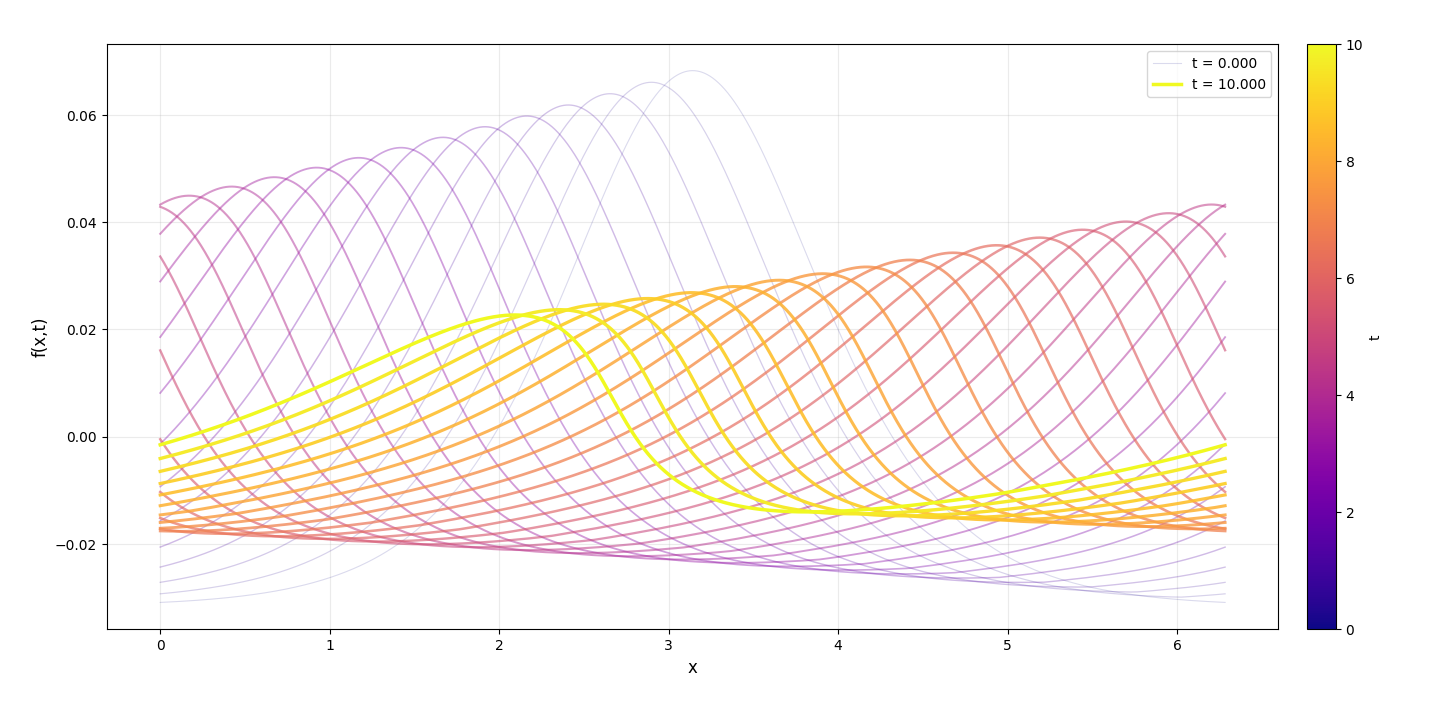}
    \caption{Numerical results for $\nu=0$ and $\beta=-1$.}
    \label{fig:beta-1}
\end{figure}

The associated diagnostics are shown in Figure~\ref{fig:beta-1-2}.

\begin{figure}[H]
    \centering
    \includegraphics[width=0.45\linewidth]{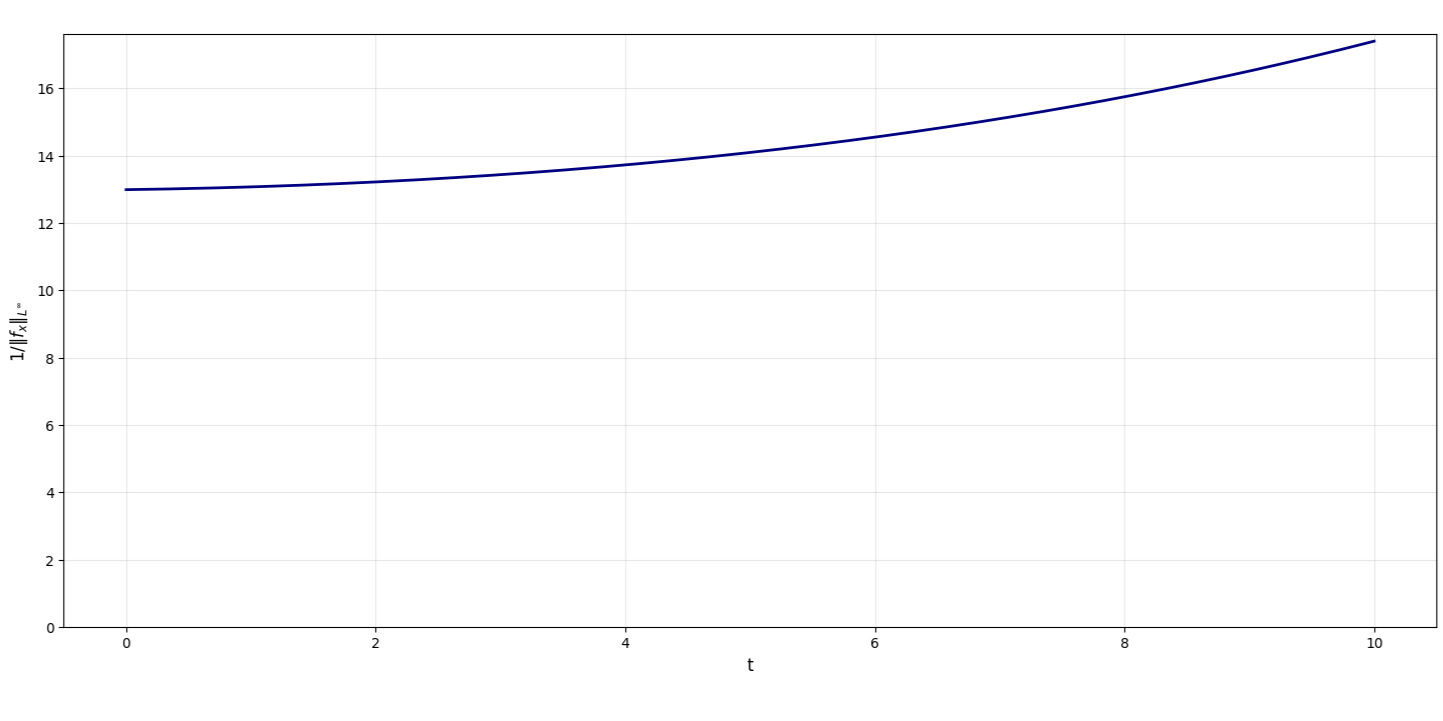}
    \includegraphics[width=0.45\linewidth]{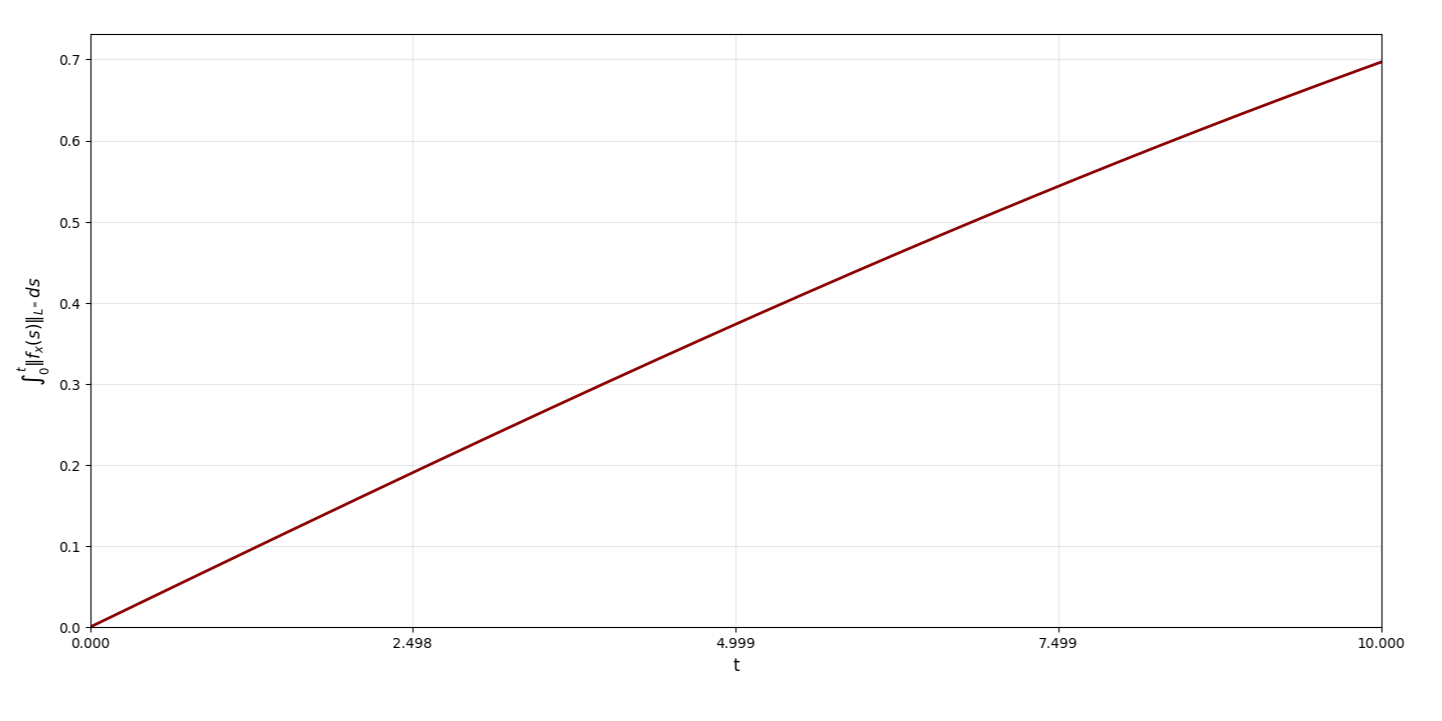}
    \caption{Diagnostics corresponding to Figure~\ref{fig:beta-1}: left, $1/\|f_x\|_{L^\infty}$; right, $I(t)$.}
    \label{fig:beta-1-2}
\end{figure}

For $\beta=-1$, the decay is less apparent over shorter time intervals.
However, when the computation is extended to $t=20$, the dissipative trend becomes more visible, as illustrated in Figures~\ref{fig:beta-1-bis} and \ref{fig:beta-1-2-bis}.

\begin{figure}[H]
    \centering
    \includegraphics[width=\linewidth]{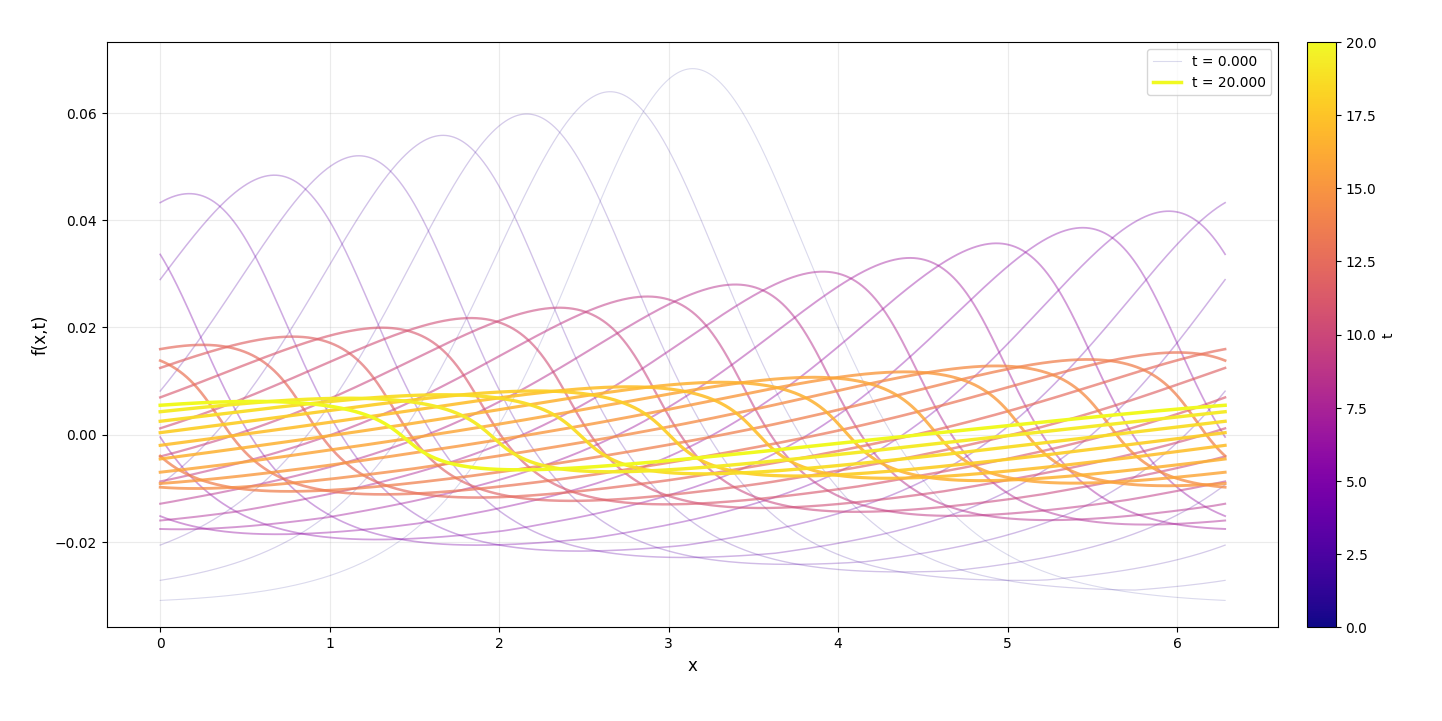}
    \caption{Numerical results for $\nu=0$, $\beta=-1$, extended up to $t=20$.}
    \label{fig:beta-1-bis}
\end{figure}

The corresponding diagnostics for the longer simulation are shown in Figure~\ref{fig:beta-1-2-bis}.

\begin{figure}[H]
    \centering
    \includegraphics[width=0.45\linewidth]{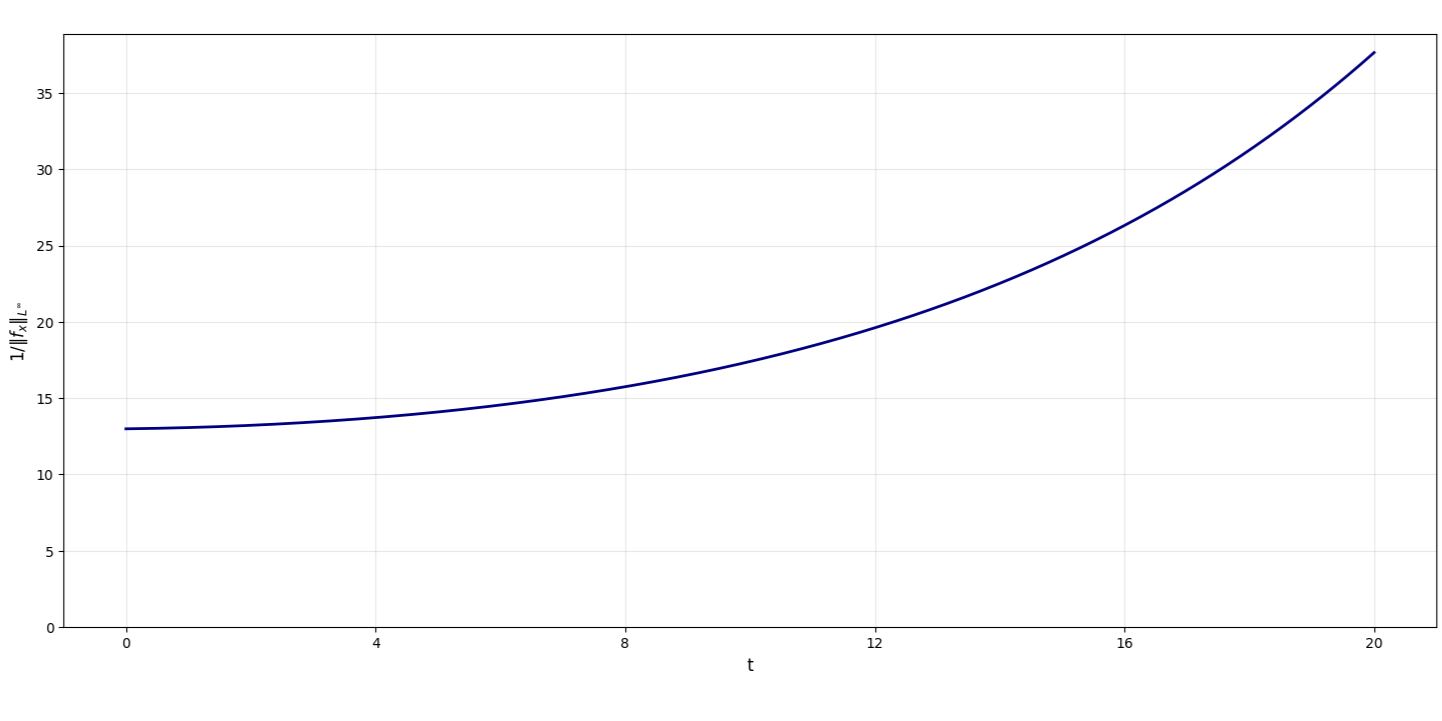}
    \includegraphics[width=0.45\linewidth]{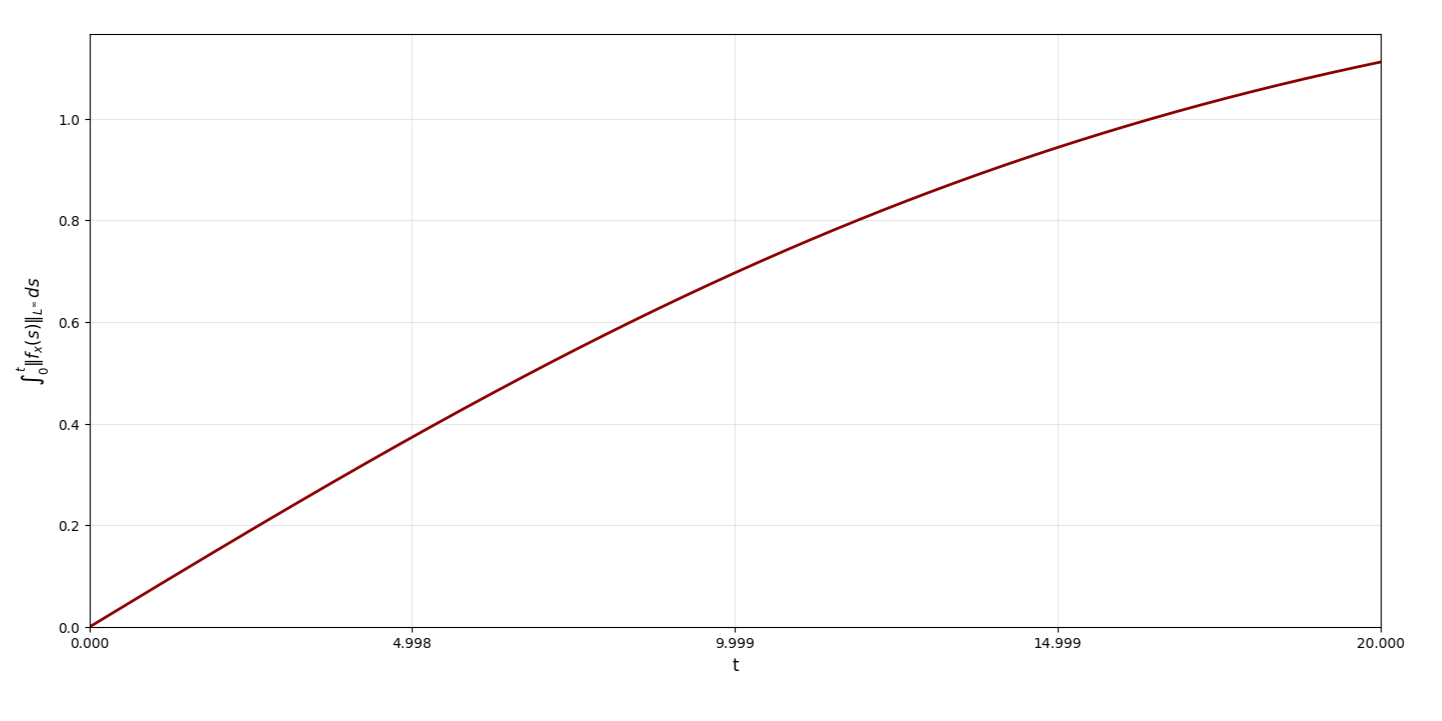}
    \caption{Diagnostics corresponding to Figure~\ref{fig:beta-1-bis}: left, $1/\|f_x\|_{L^\infty}$; right, $I(t)$.}
    \label{fig:beta-1-2-bis}
\end{figure}

\section*{Acknowledgement} 
D.A-O is supported by Grant RYC2023-045563-I funded by MICIU/AEI/10.13039/501100011033. D. A-O and R.G-B are supported by the project ”Mathematical Analysis of Fluids and Applications” Grant PID2019-109348GA-I00 and ”An\'alisis Matem\'atico Aplicado y Ecuaciones Diferenciales” Grant PID2022-141187NB-I00 funded by MCIN/AEI/10.13039/501100011033/FEDER, UE. This publication is part of the project PID2022-141187NB- I00 funded by MCIN/AEI/10.13039/501100011033.

 \end{document}